\documentclass[preprint,11pt,3p]{elsarticle}
\usepackage{amsmath,amsthm,amsfonts,amssymb}
\usepackage{xcolor}
\usepackage{algorithm}
\usepackage{multirow}
\usepackage{algpseudocode}
\newcommand{\mat}{\mathbf{mat}}
\newcommand{\blkd}{{\texttt{blkdiag}}}
\newcommand{\mc}[1]{\mathcal {#1}}
\newcommand{\m}{\!*\!{\!_M}}
\newcommand{\rg}{{\mathscr{R}}_M}
\newcommand{\mbc}{\mathbb{C}}
\newcommand{\nl}{{\mathscr{N}}_M}

\usepackage{mathrsfs}
\usepackage{sectsty}
\definecolor{astral}{RGB}{46,116,181}
\sectionfont{\color{astral}}
\newtheorem{theorem}{Theorem}[section]

\newtheorem{corollary}[theorem]{Corollary}
\newtheorem{proposition}[theorem]{Proposition}
\newcommand{\ep}{\scriptsize\mbox{\textcircled{$\dagger$}}}

\newtheorem{definition}[theorem]{Definition}
\newtheorem{example}[theorem]{Example}
\newtheorem{remark}[theorem]{Remark}
\usepackage{hyperref}
\usepackage{subfigure}

\usepackage{tikz,xcolor}
\definecolor{lime}{HTML}{A6CE39}
\definecolor{lightblue}{rgb}{0.0, 0.0, 0.5}
\DeclareRobustCommand{\orcidicon}{%
	\begin{tikzpicture}
	\draw[lime, fill=lime] (0,0)
	circle [radius=0.16]
	node[white] {{\fontfamily{qag}\selectfont \tiny ID}};
	\draw[white, fill=white] (-0.0625,0.095)
	circle [radius=0.007];
	\end{tikzpicture}
	\hspace{-2mm}
}
\foreach \x in {A, ..., Z}{%
	\expandafter\xdef\csname orcid\x\endcsname{\noexpand\href{https://orcid.org/\csname orcidauthor\x\endcsname}{\noexpand\orcidicon}}
}

\newcommand{\D}{{\mathrm D}}

\newcommand{\ind}{{\mathrm {ind}}}
\newcommand{\Ind}{{\mathrm {Ind}}}

\newcommand{\ra}[1]{\mathrm{rank}({#1})}  

\begin{document}
\begin{frontmatter}
\title{Computation of tensors generalized inverses under  $M$-product and applications
}
\author{Jajati Keshari Sahoo$^\dagger$$^a$, Saroja
  Kumar Panda$^\dagger$$^b$, Ratikanta Behera$^c$, Predrag S. Stanimirovi\'c$^{d,e}$}

\address{
 $^{\dagger}$ Department of Mathematics,
BITS Pilani K.K. Birla Goa Campus, Goa, India
\\\textit{E-mail\,$^a$}: \texttt{jksahoo@goa.bits-pilani.ac.in}
\\\textit{E-mail\,$^b$}: \texttt{p20170436@goa.bits-pilani.ac.in }\\$^{c}$Department of Computational and Data Sciences, Indian Institute of Science, Bangalore, India.\\
   \textit{E-mail\,$^c$}: \texttt{ratikanta@iisc.ac.in}\\
$^{d}$ University of Ni\v s, Faculty of Sciences and Mathematics, Ni\v s, Serbia\\
\textit{E-mail\,$^d$}: \texttt{pecko@pmf.ni.ac.rs}
}

\begin{abstract}
This paper introduces notions of the Drazin and core-EP inverses on tensors via $M$-product. We propose a few properties of the Drazin and { core-EP inverses} of tensors, as well as effective tensor-based algorithms for calculating these inverses. In addition, definitions of composite generalized inverses are presented in the framework of the $M$-product, including CMP, DMP, and MPD inverses of tensors. Tensor-based higher-order Gauss-Seidel and Gauss-Jacobi iterative methods are designed. Algorithms for these two iterative methods for solving multilinear equations have been developed. Certain multilinear systems are solved using the Drazin inverse, core-EP inverse, and composite generalized inverses, such as CMP, DMP, and MPD inverse. A tensor $M$-product-based regularization technique is applied to solve the color image deblurring.
\end{abstract}

\begin{keyword}
Generalized inverse. Drazin inverse. Core-EP inverse. Core inverse.  $M$-product.
\end{keyword}
\end{frontmatter}

\section{Introduction} \label{sec1}

Tensor generalized inverses have significantly impacted numerical multilinear algebra in handling huge volumes of multidimensional data \cite{Qi07, qi2017, Yang23}.
 Inverse and generalized inverses of tensors are closely related to multilinear systems and have been used in modeling many engineering and scientific phenomena \cite{kolda,Liang}.
Computation of tensor generalized inverses within a more general tensor product framework and computationally efficient algorithms becomes particularly challenging.
Recently, there has been a growing interest in developing multidimensional data analysis theory and computations of tensors (see, e.g., results related to the $t$-product \cite{kilmer11,kilmer13}, cosine transform product ($c$-product) \cite{Kernfeldlinear}, $M$-product \cite{Kernfeldlinear}, Einstein product \cite{ein}, Shao's general product \cite{Shao13} and the $n$-mode product \cite{qi2017}, etc.)
These tensor products have also been the subject of intense mathematical research and applications in diverse fields, which has attracted considerable attention in recent literature, e.g., computer vision \cite{hao2013,hu16}, image processing \cite{martin13,soltani16,tarzanagh2018}, signal processing \cite{chan16,liu18}, control systems \cite{Wangwei22}, face recognition \cite{kilmer13}, data completion and denoising \cite{hu16,Wang19}, low-rank tensor recovery \cite{Wang19}, and robust tensor  Principal Component Analysis (PCA) \cite{liu18}.

In view of the above theory and computations of the tensor product, Brazell {\it et al.} \cite{Brazell} discussed the ordinary tensor inverse and solution of multilinear systems in the outline of the Einstein product.
Subsequently, many authors worked on tensor generalized inverses via the Einstein product (e.g., the outer and $(b,c)$ inverses \cite{Stanimirovic}, Moore-Penrose (MP) inverse \cite{SunEinstein,Behera}, core and core-EP inverses \cite{Sahoo}, and the Drazin inverse \cite{Jidrazin,AsishRJ}).
In addition, many researchers have studied composite generalized inverses via the Einstein product (e.g., DMP and CMP inverse \cite{WangCMP20}, MPD inverse \cite{Ma1filomat9}, and bilateral inverses \cite{behera2023computing,PublishedSalemi23}).

In addition to previous research, tensor calculus based on t-product, introduced by Kilmer \&  Martin (2011), has significantly impacted in computational and applied mathematics \cite{kilmer11,kilmer13}.
Following this work, the MP inverse \cite{Jin17}, outer inverse \cite{behera2023computation}, Drazin inverse \cite{AsishRJ}
and singular value decomposition \cite{hao2013, Miaogeneralized,Wangneural}, of tensors have been developed.
In 2015, the authors of \cite{Kernfeldlinear} discussed tensor–tensor products with invertible linear transforms based on the $M$-product and $c$-product of tensors.
In \cite{kilm21}, the $M$-product tensor based generalized inverses compose a relatively new research field that recently appeared in the literature. The main results referring to the MP inverse of tensors via the $M$-product are available in \cite{jin2023}. To the best of our knowledge, only the MP inverse of tensors has been discussed via $M$-product in the available literature. Our research aims to design the core-EP and Drazin inverse of tensors in the framework of $M$-product.

Recently, there has been growing interest in developing rapid and efficient tensor-based numerical methods for solving multilinear systems. However, solving multilinear systems depends on the different structures of tensor computations. Following the Einstein product, the authors of \cite{Brazell,Wang19,AsishRJ} discussed iterative methods for solving several tensor equations \cite{Wang19}. At the same time, following the $t$-product, the authors of \cite{Wang2023} proposed a randomized Kaczmarz-like algorithm and its relaxed version, and solved the linear complementarity problem using a tensor-based fixed point iterative method. Further, several iterative methods have been discussed to calculate approximate solutions to ill-posed systems, i.e., preconditioned LSQR (see \cite{kilmer11}), conjugate gradient (CG), $t$-singular value decomposition ($t$-SVD), and Golub–Kahan iterative bidiagonalizatios (G-K-Bi-diag) \cite{kilmer13,GuideKy22}. The $M$-product is a class of tensor products that generalizes the notions of the $t$-product and the $c$-product.

A brief summary of the main contributions of our study is listed below.

\begin{itemize}
\item The Drazin and core-EP inverses of the tensors via the $M$-product are introduced, and effective algorithms for computing these inverses are presented.
\item The main properties, characterizations, and representations of the Drazin and core-EP inverse are established.
\item The definitions of composite generalized inverses are presented in the framework of $M$-product, including CMP, DMP, and MPD inverses of tensors.
\item Higher order Gauss-Seidel and Gauss-Jacobi iterative methods have been developed to solve multilinear equations.
\item Solutions for particular multilinear systems in terms of the Drazin inverse, core-EP inverse, and composite generalized inverses are considered.
\item The application of the proposed $M$-product-based regularization technique to color image deblurring is considered.
\end{itemize}

The overall plan for the current research is subject to the next organization. Section \ref{sec:pre} is presents the notations and definitions of tensors in the environment of the $M$-product, which will support to introduce the main results. Section \ref{SecDrMt} presents the Drazin inverse of tensors and discusses a few characterizations of these inverses using  $M$-product. On the other hand, in Section 4, we introduce the tensor core-EP inverse and design an algorithm for its computations. Composite generalized inverses are discussed in section 5. Iterative solutions to multilinear systems, Tikhonov’s regularization, and color image deblurring problems are discussed in Section 6. Finally, the conclusions and possible future research directions are mentioned in Section 7.

\section{Basic terms and preliminaries}\label{sec:pre}

The $i$th frontal slice of  tensor $\mc{A}\in\mbc^{m \times n \times p}$ is denoted by $\mc{A}^{(i)}=\mc{A}(:,:,i)$.
The tube fibers of $\mc{A}$ are labeled with either $\mc{A}(i, j,:)$ or $ \mc{A}(i,:,k)$ or $ \mc{A}(:,j,k)$.
Elements of $\mc{A}$ are denoted either by $(\mc{A})_{ijk}$ or $a_{ijk}$, such that $i = \overline{1,m}$, $j = \overline{1,n}$; and $k = \overline{1,p}$, 
 { assuming that the notation $l=\overline{1,s}$ stands for for $l = 1,2, \ldots, s$.}

\begin{definition}{\rm \cite{kolda}}
   Let $\mc{A}\in {\mbc}^{m\times n \times q}$ be a tensor and $B\in{\mbc^{p\times q}}$ be a matrix.
   The $3$-mode product of $\mc{A}$ with $B$ is denoted by $\mc{A}\times_{3}B\in{\mbc}^{m\times n\times p}$ and element-wise defined by
      \begin{equation*}
          (\mc{A}\times_{3} B)_{ijk}=\sum_{s=1}^{q}a_{ijs}b_{ks},\quad i=\overline{1,m}, ~j=\overline{1,n},~k=\overline{1,p}.
      \end{equation*}
\end{definition}

In particular, if $\mc{A}\in \mbc^{m\times n\times p}$ and $M\in \mbc^{p\times p}$ is invertible, the ``tilde'' notation denotes a tensor in the transform domain predetermined by $M$, that is,
\begin{equation*}
    \tilde{\mc{A}}:=\mc{A}\times_{3}M \in \mbc^{m\times n \times p}.
\end{equation*}
The ``tilde'' notation should be recognized in the context relative to the applied $M$.
From here onwards, we will consider $M$ to be nonsingular. Next, we define the $\mat$ operation to transform a tensor into a matrix with respect to $M$.

\begin{definition}[Definition 2.6 and Lemma 3.1, \cite{Kernfeldlinear}]\label{Kernfeldlinearrati}
Let $\mc{A}\in \mbc^{m\times n\times p}$,  $M\in \mbc^{p\times p}$ and $(\tilde{\mc{A}})^{(i)}\in \mbc^{m\times n}$ be the $i$th frontal slice of $\tilde{\mc{A}}$, for $i=\overline{1,p}$.
Then $\mathrm{\mat}:\mbc^{m\times n\times p} \mapsto \mbc^{mp\times np}$ is defined as
\[\mat(\mc{A})=\begin{bmatrix}
(\tilde{\mc{A}})^{(1)} &0 &\cdots &0\\
0 & (\tilde{\mc{A}})^{(2)}  & \cdots &0\\
\vdots  &   \vdots  &    \ddots & \vdots\\
0& 0&  \cdots & (\tilde{\mc{A}})^{(p)}
\end{bmatrix}.
\]
\end{definition}

The inverse operation is defined by the following algorithm.
\begin{algorithm}[H]
 \caption{$\mat^{-1}(A)$}\label{alg:matinvINV}
\begin{algorithmic}[1]
\State {\bf Input} ${A} \in \mbc^{mp \times np}$ and $M\in\mbc^{p\times p}$
\For{$i \gets 1$ to $p$}
   \State $\mc{B}(:,:,i) = A((i-1)m+1:im,(i-1)n+1:in)$
   \EndFor
\State Compute $\mc{A}=\mc{B}\times_3M^{-1}$
\State \Return $\mc{A}=\mat^{-1}(A)$
 \end{algorithmic}
\end{algorithm}

{ 
In view of  Algorithm \ref{alg:matinvINV} and Definition \ref{Kernfeldlinearrati}, every tensor $\mc{A}\in\mbc^{m\times n\times p}$ is represented as
\[\mc{A}=\mat^{-1}(\mat(\mc{A})).\]}

Note that the $t$-product in \cite{kilmer13} is obtained by choosing the matrix $M$ as the unnormalized DFT matrix.
The cosine transform product, called the $c-$product, \cite{kilm21} is obtained by taking $M=W^{-1}C(I+Z)$, where $C$ is the DCT matrix of order $n$, $Z=\mathrm{diag}(ones(n-1,1),1)$ and $W=\mathrm{diag}(C(:, 1))$.

Now, the tensor product is defined with the support of the operators `$\mat$' and `$\mat^{-1}$', as in Definition \ref{DeftProd}.

 \begin{definition}\label{DeftProd}{\rm \cite{Kernfeldlinear}}
 For a given  $M\in \mbc^{p\times p}$, the $M$-product of tensors $\mc{A}\in \mbc^{m\times n\times p}$ and $\mc{B}\in \mbc^{n\times k\times p}$ is denoted by $\mc{A}*_{M}\mc{B}\in \mbc^{m\times k\times p}$ and defined by
\[\mc{A}*_{M}\mc{B}=\mat^{-1}(\mat(\mc{A})\mat(\mc{B})).\]
\end{definition}

\begin{example}\rm
Let $\mc{A}, \mc{B}\in \mbc^{3\times 3\times 3}$ be defined by the entries
\[\mc{A}(:,:,1)=\begin{bmatrix}
 1   &  0 &    1\\
    -1  &   2  &   3\\
     0  &   1  &   0
\end{bmatrix},~\mc{A}(:,:,2)=\begin{bmatrix}
  2   &  3 &   -1\\
    -1  &   0  &   3\\
     1   &  1   &  1
\end{bmatrix},~\mc{A}(:,:,3)=\begin{bmatrix}
   0    & 0   &  1\\
     3   &  0  &   1\\
    -1  &   0   & -1
\end{bmatrix},\]
\[\mc{B}(:,:,1)=\begin{bmatrix}
2   & -1  &   0\\
	     1   &  2  &   1\\
	    -1 &    1   &  0
\end{bmatrix},~\mc{B}(:,:,2)=\begin{bmatrix}
 1  &   -1   &   1\\
	     3   &   0   &   1\\
	     0    &  1  &    0
\end{bmatrix},~\mc{B}(:,:,3)=\begin{bmatrix}
   2   &   1    & -1\\
	     1     & 0   &   0\\
	     1   &   0     & 2
\end{bmatrix},\]
and $M=\begin{bmatrix}
2 & 2 & 3 \\2 & 3 & 1 \\1 & 1 & 1
\end{bmatrix}$.
The product $\mc{C}=\mc{A}*_{M}\mc{B}$ is computed as
\[\mc{C}(:,:,1)=\begin{bmatrix}
 -252 &  -29 &  -77\\
	  -147 & -127 & -124\\
	   -68  & -20  & -27
\end{bmatrix},~\mc{C}(:,:,2)=\begin{bmatrix}
 201 &   14  &  65\\
	    73  & 107  &  77\\
	    67   & 19   & 28
\end{bmatrix},~\mc{C}(:,:,3)=\begin{bmatrix}
 81   & 20  &  20\\
	    89 &   37 &   65\\
	    11   &  5   &  3
\end{bmatrix}.\]
\end{example}

\begin{definition}{\rm \cite{Kernfeldlinear}}
If $M\in\mbc^{p\times p}$ is fixed, then $\mc{A}\in \mbc^{m\times m\times p}$ is said to be the identity tensor if $(\mc{A}\times_{3}M)^{({i})}=I_m$ for $i=\overline{1,p}$, where $I_m$ is the $m\times m$ identity matrix.
Similarly, $\mc{O}\in \mbc^{m\times n\times p}$ is a zero or null tensor if $(\mc{A}\times_{3}M)^{({i})}$ are zero matrices for each $i=\overline{1,p}$.
\end{definition}

\begin{definition}
 The tensor $\mc{A} \in \mbc^{m\times n\times p}$ is called diagonally dominant with respect to $M\in\mbc^{p\times p}$ if all frontal slices of $\tilde{\mc{A}}=\mc{A}\times_{3}M$ are diagonally dominant.
\end{definition}

\begin{definition}{\rm \cite{Kernfeldlinear}}
The transpose conjugate of $\mc{A}\in \mbc^{m\times n\times p}$ based on $M\in\mbc^{p\times p}$ is denoted by $\mc{A}^{H}$ and is defined by the relation  $(\mc{A}^{H}\times_{3}M)^{(i)}= (\tilde{\mc{A}}^{(i)})^{H}$, $i=\overline{1,p}$, { where $\tilde{\mc{A}}=\mc{A}\times_{3}M$}.
\end{definition}

\begin{definition}{\rm \cite{Kernfeldlinear}}
{ The tensor} $\mc{A}\in \mbc^{m\times m\times p}$ is called invertible in the tensor product determined by $M\in\mbc^{p\times p}$ if there exists a tensor $\mc{X}\in\mbc^{m\times m\times p}$
such that both $\mc{A}\m \mc{X}$ and $\mc{X}\m \mc{A}$ are identity tensors. If $\mc{A}$ is not invertible, it is termed as a singular tensor.
\end{definition}

It is easy to see that both $\mat$ and $\mat^{-1}$  are linear and isomorphic maps in their domains.
Additional properties on the $\mat$ and $\mat^{-1}$ operations can be found \cite{behera2024}.

\begin{definition}
The range and null space of $\mc{A} \in \mbc^{m\times n\times p}$ relative to $M\in\mbc^{p\times p}$ are defined as
$$
\aligned
\rg(\mc{A})&=\{\mc{A}\m\mc{Z}~:~{ \mathcal{Z}}\in \mbc^{n\times 1\times p}\}\subseteq \mbc^{m\times 1\times p},\\
 \nl(\mc{A})&=\{\mc{Y}~:~\mc{A}\m\mc{Y}=\mc{O}\in \mbc^{m\times 1\times p} \}\subseteq \mbc^{n\times 1\times p}.
 \endaligned
 $$
\end{definition}

The following proposition will be useful in proving the results related to the range and null space.

\begin{proposition}\label{proprn9}
 Let $M\in\mbc^{p\times p}$ and $\mc{A},~\mc{B} \in \mbc^{m\times n\times p}$.
 Then $\rg(\mc{A})\subseteq\rg(\mc{B})\Longleftrightarrow\mc{A}=\mc{B}\m\mc{T}$ for some $\mc{T}\in \mbc^{n\times n\times p}$.
\end{proposition}

\begin{proof}
Let us assume $\mc{A}=\mc{B}\m\mc{T}$ and choose $\mc{Y}\in\rg(\mc{A})$.
  Then, $\mc{Y}=\mc{A}\m\mc{Z}$ for some $\mc{Z}\in \mbc^{n\times 1\times p}$. Now
    \[\mc{Y}=\mc{A}\m\mc{Z}=\mc{B}\m\mc{T}\m\mc{Z}=\mc{B}\m\mc{T}_1,\ \ \mc{T}_1=\mc{T}\m\mc{Z}\in \mbc^{n\times 1\times p}.\]
    Thus $\mc{Y}\in \rg(\mc{B})$.
Conversely, if $\rg(\mc{A})\subseteq\rg(\mc{B})$ then for all ${ \mc{Z}}\in \mbc^{n\times 1\times p}$ it follows that
    \begin{equation*}
\mc{A}\m\mc{Z}\in\rg(\mc{A})\subseteq\rg(\mc{B})=\mc{B}\m\mc{T} \text{~~for some~~} \mc{T}\in\mbc^{n\times 1\times p}
   \end{equation*}
Taking the frontal slices of $\mc{Z}$ as the frontal slices of the identity tensor, the result can be concluded.
\end{proof}

Next, we discuss the tubal spectral norm (or tubal norm) and the multirank \cite{Kernfeldlinear, kilm21} of a tensor.

\begin{definition}
Let  $\mc{A} \in \mbc^{m\times n\times p}$, $M\in\mbc^{p\times p}$ { and $\tilde{\mc{A}}=\mc{A}\times_{3}M$}. Then
\begin{enumerate}
\item[\rm (i)] the tubal norm of $\mc{A}$ is defined as
$\|\mc{A}\|_M=\displaystyle{\max_{1\leq i\leq p}(\|\tilde{\mc{A}}(:,:,i)\|_2)}, \text{~where~} \|\cdot\|_2$ is the spectral matrix norm.
\item[\rm (ii)] the multirank of $\mc{A}$ is denoted by $r_M(\mc{A})$ and defined as $r_M(\mc{A})=(r_1,\ldots,r_p)$ where $r_i=\ra{\tilde{\mc{A}}(:,:,i)}$, $i=\overline{1,p}.$
 \end{enumerate}
\end{definition}

The multi index and tubal index of a tensor, are defined below.

\begin{definition}
The multi index of $\mc{A} \in \mbc^{m\times m\times p}$ induced by  $M\in\mbc^{p\times p}$ is denote by $\Ind_M(\mc{A})$ and defined as
 $\Ind_M(\mc{A})=(k_1,\ldots,k_p)$ where $k_i=\ind(\tilde{\mc{A}}(:,:,i))$, $i=\overline{1,p}$ { and $\tilde{\mc{A}}=\mc{A}\times_{3}M$}.
 The standard notation $\ind(A)$ represents the minimal non-negative integer $k$ determining the rank-invariant power $\mathrm{rank}(A^{k})=\mathrm{rank}(A^{k+1})$. The tubal index $k$ is defined as $k:=\max_{1\leq i\leq p}k_i$.
 \end{definition}

\begin{remark}
   The tubal rank (or $t$-rank in {\rm \cite{kilmer13}}) of $\mc{A} \in \mbc^{m\times n\times p}$ is defined as $\displaystyle\max_{1\leq i\leq p}r_i$, where $r_i=\mathrm{rank}(\mc{A}(:,:,i)$).
\end{remark}

\begin{definition} \label{tenseries}
Let $\mc{A}\in\mathbb{C}^{m \times m \times p}$,  $M\in\mathbb{C}^{p \times p}$ { and $\tilde{\mc{A}}=\mc{A}\times_{3}M$}.
\begin{enumerate}
    \item[\rm (i)] The series $\displaystyle\sum_{s=0}^\infty \mc{A}^{s}$ is convergent, if  $\displaystyle\sum_{s=0}^\infty [\tilde{\mc{A}}(:,:i)]^s$ is convergent for each $i$, where $i=\overline{1,p}$.
    \item[\rm (ii)] $\displaystyle{\lim_{s \to\infty} \mc{A}^s}=\mc{L}$ if  $\displaystyle{\lim_{s \to\infty}({\tilde{\mc{A}}}^{s})_{ijl}}=(\mc{L})_{ijl}$ for $1\leq i,j\leq m$ and $l=\overline{1,p}$.
\end{enumerate}
\end{definition}

\begin{definition}
    Let $\mc{A}\in \mbc^{m\times m \times p}$ and $M \in \mbc^{p\times p}$. If  $\mc{X} \neq \mc{O} \in \mbc^{m\times 1 \times n}$  satisfy
  \[\mc{A}\m \mc{X} =\lambda\mc{X},~~\lambda \in \mbc.\]
Such $\lambda$ is termed as an $M$-eigenvalue of $\mc{A}$ and $\mc{X}$ is the $M$-eigenvector of $\mc{A}$ based on $M$ and $\lambda$.
  Further, the spectral radius of $\mc{A}$ is denoted as $\rho(\mc{A})$ and is defined as
${\displaystyle \rho (\mc{A})=\max_{1\leq i\leq p}\{\rho (\tilde{\mc{A}}(:,:,i})\}$.
\end{definition}

The MP inverse under the $M$-product is introduced in \cite{jin2023}.
The tensor $\mc{X}\in \mbc^{n\times m\times p}$ is the MP inverse of $\mc{A}\in \mbc^{m\times n\times p}$ with an underlying $M\in\mbc^{p\times p}$ if it satisfies
\begin{equation*}
\mc{A}\m\mc{X}\m\mc{A}=\mc{A}, ~\mc{X}\m\mc{A}\m\mc{X}=\mc{X},~\mc{A}\m\mc{X}=(\mc{A}\m\mc{X})^H \text{~and~} \mc{X}\m\mc{A}=(\mc{X}\m\mc{A})^H.
\end{equation*}
The MP inverse  of $\mc{A}$ is denoted by $\mc{A}^{\dagger}$ and a few properties of the MP can be found in \cite{jin2023}.

\section{Drazin inverse of tensors under $M$-product}\label{SecDrMt}

In this section, we recall the definition and discuss the main properties of the Drazin inverse in the tensor structure domain within the framework of the $M$-product.
In addition, we also show the computational efficiency of a test example.

\begin{definition}\label{drdef}
Let $M\in\mbc^{p\times p}$ and $\mc{A} \in \mbc^{m\times m\times p}$ be of the tubal index $k$.
The Drazin inverse $\mc{A}^\D$ of $\mc{A}$ is defined as a unique tensor $\mc{X}\in \mbc^{m\times m\times p}$ satisfying $\mc{X}\m\mc{A}\m\mc{X}=\mc{X}$, $\mc{A}\m\mc{X}=\mc{X}\m\mc{A}$ and $\mc{X}\m\mc{A}^{k+1}=\mc{A}^k$.
\end{definition}

We can also compute the Drazin inverse using  $\mat$ and $\mat^{-1}$, as stated below.

\begin{proposition}\label{propdr}{\rm \cite{behera2024}}
Let $M\in\mbc^{p\times p}$ and $\mc{A} \in \mbc^{m\times m\times p}$ with  $\ind(\mat(\mc{A})) = k$. Then,
\[\mc{A}^{\D} = \mat^{-1}(\mat(\mc{A})^{\D}).\]
\end{proposition}

\begin{remark}\label{rm1.2}
In Definition \ref{drdef}, it is possible to compute the Drazin inverse using a multi index, where each frontal slice of the Drazin inverse is computed using the respective index. Further, in Proposition \ref{propdr}, the computation is slower because of the matrix-structured domain.
\end{remark}

In view of Definition \ref{drdef}, we design Algorithm \ref{alg:mdrazin} for calculating the Drazin inverse in the $M$-product framework.

\begin{algorithm}[H]
 \caption{Computing the Drazin inverse under $M$-product} \label{alg:mdrazin}
\begin{algorithmic}[1]
  \Procedure{Drazin inverse} {$\mc{A}^\D$}
\State {\bf Input} $\mc{A} \in \mbc^{m \times m\times p}$ and  $M\in\mbc^{p\times p}$.
\State Compute $\tilde{\mc{A}}=\mc{A}\times_3 M$
\For{$i \gets 1$ to $p$}
\State $k_i=\ind(\tilde{\mc{A}}(:,:,i))$
   \EndFor
   \State Compute $k=\max_{1\leq i\leq p}k_i$
\For{$i \gets 1$ to $p$}
   \State $\mc{Z}(:,:,i) = (\tilde{\mc{A}}(:,:,i))^\D$ 
   \EndFor
\State Compute $\mc{X}=\mc{Z}\times_3M^{-1}$
\State \Return $\mc{A}^\D=\mc{X}$
\EndProcedure
 \end{algorithmic}
\end{algorithm}

Applying Algorithm \ref{alg:mdrazin}, we compute the Drazin inverse for the following  example.

\begin{example}\rm\label{exa3.4}
  Let $\mc{A}\in\mbc^{3\times 3\times 3}$ with entries
  \[\mc{A}(:,:,1)=\begin{bmatrix}
  4   & -4  &  -1\\
	    -7  &  -8  &   7\\
	    -1   & -2   &  0
  \end{bmatrix},~\mc{A}(:,:,2)=\begin{bmatrix}
  -2   &  2   &  1\\
	     4   &  4   & -4\\
	     0  &   1  &   0
  \end{bmatrix},~\mc{A}(:,:,3)=\begin{bmatrix}
   -1    & 2  &   0\\
	     3  &   4   & -2\\
	     1    & 1  &   0
  \end{bmatrix},~M=\begin{bmatrix}
  2& 2 &3\\
  2 &3 &1\\
  1& 1& 1
\end{bmatrix}.\]
We evaluate the tubal index $k=2=\max\{1,1,2\}$ since $\ind(\tilde{\mc{A}}(:,:,1))=1=\ind(\tilde{\mc{A}}(:,:,2))$, $\ind(\tilde{\mc{A}}(:,:,3))=2$.
Using Algorithm \ref{alg:mdrazin}, we calculate $\mc{X}=\mc{A}^\D$, where
\[\mc{X}(:,:,1)=\begin{bmatrix}
  -11  &  -2 & 1\\
	   -2&  1  &  1\\
	   -4.5& -1.5 &  -1
\end{bmatrix},~\mc{X}(:,:,2)=\begin{bmatrix}
 -6  & 1 &  -1\\
	  0.5 &  -0.5&  -1\\
	  2.75 & -0.75 & 1
\end{bmatrix},~\mc{X}(:,:,3)=\begin{bmatrix}
 -4&  1&    0\\
	         1.5    &     -0.5  &0\\
	    1.75 & -0.75 &  0
\end{bmatrix}.\]
\end{example}

Table \ref{tab:error} surveys the errors associated with tensor equations underlying the definitions of the generalized inverses of $\mc{A}$.
\begin{table}[H]
\begin{center}
\caption{Errors associated with different tensor equations}\label{tab:error}
 \vspace*{0.2cm}
    \small{
    \renewcommand{\arraystretch}{1.2}
    \begin{tabular}{|c|c|c|}
    \hline
     $\mc{E}^{M}_{1} = \|\mc{A}-\mc{A}\m \mc{X}\m \mc{A}\|_M$  & $\mc{E}^M_{2} = \|\mc{X}-\mc{X}\m \mc{A}\m \mc{X}\|_M$ & $\mc{E}^M_{3} =\|\mc{A}\m \mc{X}-(\mc{A}\m \mc{X})^T\|_M$ \\
     \hline
   $\mc{E}^{M}_{5} = \|\mc{A}\m \mc{X}-\mc{X}\m \mc{A}\|_M$
       &$\mc{E}^{T}_{1^k} = \|\mc{X}\m \mc{A}^{k+1} -\mc{A}^k\|_M$&$\mc{E}^{M}_{7} = \|\mc{A}\m \mc{X}^2-\mc{X}\|_M$ \\
     \hline
    \end{tabular}}
       \end{center}
\end{table}

The norms arranged in the Table \ref{tab:error} can be used under other products generated by changing the matrix $M$.
For example $\mc{E}_1^t$ means the error calculated by taking $M$ as the DFT matrix.
Next, we discuss the comparative analysis of the mean CPU time (MT$^M$) and errors (Error$^M$) obtained during the calculation of the Drazin inverse for different choices of $M$.
The tensor $\mc{A}$ { in Table} \ref{tab:comp-draz-m} is chosen randomly so that it has an index of $1$ or $2$.

\begin{table}[H]
    \begin{center}
       \caption{Computational time for computing $\mc{A}^D$ for different tensor products}
        \vspace{0.2cm}
        \renewcommand{\arraystretch}{1.1}
        \small
        \begin{tabular}{cccccccc}
    \hline
        Size of $\mc{A}$& $k$  & MT$^{t}$  & MT$^{c}$ & MT$^M$ & Error$^{t}$ & Error$^{c}$ &Error$^M$\\
           \hline
    \multirow{4}{*}{$300\times 300\times 300$}& \multirow{4}{*}{1} &  \multirow{4}{*}{34.18} &\multirow{4}{*}{14.14}&\multirow{4}{*}{11.02}&
    $\mc{E}^{t}_{1}= 1.09e^{-07}$  & $\mc{E}^{c}_{1}=1.23e^{-07}$ & $\mc{E}^{M}_{1}=3.91e^{-07}$
    \\
& & & & & $\mc{E}^{t}_{2}= 3.90e^{-08}$  & $\mc{E}^{c}_{2}=3.45e^{-08}$ & $\mc{E}^{M}_{2}=5.75e^{-07}$
    \\
          &  & & & & $\mc{E}^{t}_{5}= 1.36e^{-09}$  & $\mc{E}^{c}_{5}=1.54e^{-09}$ & $\mc{E}^{M}_{5}=4.00e^{-09}$
       \\
           \hline
    \multirow{4}{*}{$400\times 400\times 400$}& \multirow{4}{*}{1} &  \multirow{4}{*}{50.80} &\multirow{4}{*}{29.46}&\multirow{4}{*}{28.18}&
     $\mc{E}^{t}_{1}= 3.87e^{-08}$  & $\mc{E}^{c}_{1}=1.874e^{-07}$ & $\mc{E}^{M}_{1}=2.11e^{-07}$
    \\
& & & & & $\mc{E}^{t}_{2}= 3.47e^{-08}$  & $\mc{E}^{c}_{2}=1.76e^{-08}$ & $\mc{E}^{M}_{2}=3.89e^{-07}$
    \\
          &  & & & & $\mc{E}^{t}_{5}= 1.02e^{-10}$  & $\mc{E}^{c}_{5}=3.43e^{-10}$ & $\mc{E}^{M}_{5}=1.07e^{-08}$
       \\
           \hline
    \multirow{4}{*}{$300\times 300\times 300$}& \multirow{4}{*}{2} &  \multirow{4}{*}{35.09} &\multirow{4}{*}{16.26}&\multirow{4}{*}{15.75}&
    $\mc{E}^{t}_{1^2}= 3.49e^{-07}$  & $\mc{E}^{c}_{1^2}=1.74e^{-05}$ & $\mc{E}^{M}_{1^2}=9.22e^{-08}$
    \\
& & & & & $\mc{E}^{t}_{2}= 4.44e^{-12}$  & $\mc{E}^{c}_{2}=3.70e^{-10}$ & $\mc{E}^{M}_{2}=5.19e^{-09}$
    \\
          &  & & & & $\mc{E}^{t}_{5}= 2.74e^{-11}$  & $\mc{E}^{c}_{5}=1.32e^{-10}$ & $\mc{E}^{M}_{5}=8.72e^{-10}$
       \\
           \hline
    \multirow{4}{*}{$400\times 400\times 400$}& \multirow{4}{*}{2} &  \multirow{4}{*}{51.72} &\multirow{4}{*}{38.93}&\multirow{4}{*}{38.92}&
    $\mc{E}^{t}_{1^2}= 4.32e^{-06}$  & $\mc{E}^{c}_{1^2}=9.64e^{-05}$ & $\mc{E}^{M}_{1^2}=4.08e^{-07}$
    \\
& & & & & $\mc{E}^{t}_{2}= 2.14e^{-11}$  & $\mc{E}^{c}_{2}=3.70e^{-10}$ & $\mc{E}^{M}_{2}=1.19e^{-09}$
    \\
          &  & & & & $\mc{E}^{t}_{5}= 6.55e^{-11}$  & $\mc{E}^{c}_{5}=4.13e^{-10}$ & $\mc{E}^{M}_{5}=1.29e^{-09}$
    \\
    \hline
    \end{tabular}
    \label{tab:comp-draz-m}
    \end{center}
    \end{table}

Example \ref{exa-com} is developed in view of Remark \ref{rm1.2}.

\begin{example}\label{exa-com}
Consider  $M=\mathrm{rand}(n)$ and  randomly choose $\mc{A} \in \mbc^{n \times n\times n}$ with tubal index equal to $1$ or $2$.
A comparison of the mean CPU time (MT) corresponding to the randomly generated invertible matrix $M$, under tubal index and $\mat$ operation is presented in Table \ref{tab:comp-draz}.
\begin{table}[H]
    \centering
    \caption{Comparison of mean CPU times for computing $\mc{A}^{\D }$}
     \vspace{0.2cm}
        \renewcommand{\arraystretch}{1.1}
    \begin{tabular}{|ccc|ccc|}
    \hline
      \scriptsize{Size of $\mc{A}$}   & $k$ & \scriptsize{MT (Using tubal index)} & \scriptsize{Size of $\mat(\mc{A})$} & \scriptsize{$ind(\mat(\mc{A}))$} &\scriptsize{MT (Using $\mat$ and $\mat^{-1}$)} \\
         \hline
   \scriptsize{$60\times 60\times 60$}  & 1 &0.19 & \scriptsize{$3600\times 3600$}& 1 & 8.10  \\
      \hline
        \scriptsize{$80\times 80\times 80$} & 1 &0.37 &\scriptsize{$6400\times 6400$}& 1 & 39.37   \\
      \hline
         \scriptsize{$100\times 100\times 100$}  &2  &0.94 &\scriptsize{$10000\times 10000$}& 2&  169.72\\
      \hline
        \scriptsize{$120\times 120\times 120$}  &2  & 1.60  & \scriptsize{$14400\times 14400$} &2&  434.46\\
      \hline
    \end{tabular}
        \label{tab:comp-draz}
\end{table}
\end{example}

\subsection{Characterizations of the Drazin inverse}
The results stated in Proposition \ref{PropTD1} can be verified in light of the definition of the Drazin inverse.
\begin{proposition}\label{PropTD1}
Let $M\in\mbc^{p\times p}$ and $\mc{A} \in \mbc^{m\times m\times p}$ with  tubal index $k$. Then
\begin{enumerate}
    \item[\rm (i)] $(\mc{A}\m\mc{A}^\D)^s=\mc{A}\m\mc{A}^\D=\mc{A}^s\m(\mc{A}^\D)^s$ and $(\mc{A}^{s})^{\D} =(\mc{A}^{\D})^{s}$ for all $s\in\mathbb{N}$.
    \item[\rm (ii)] $(\mc{A}^*)^\D=(\mc{A}^\D)^*$.
    \item[\rm (iii)] $((\mc{A}^{\D})^{\D})^{\D}=\mc{A}^{\D}.$
    \end{enumerate}
\end{proposition}

A tensor $\mc{A}\in\mbc^{m\times m\times p}$ is nilpotent if $(\mc{A}\times_3 M)^{(i)}$ is a nilpotent matrix, for each $i,~i=\overline{1,p}$.
If the nilpotent index of respective frontal slices $(\mc{A}\times_3 M)^{(i)}$ is $n_i (\leq m )$ and $n=\max_{1\leq i\leq p}n_i$ then we can also say that $\mc{A}$ is nilpotent if $\mc{A}^n=\mc{O}$, where $\mc{O}$ is the null tensor.
In this case, $n$ is the nilpotent index of $\mc{A}$.

\begin{example}\rm\label{ex-nil}
  Let $\mc{A}\in\mbc^{3\times 3\times 3}$ be determined by the entries
  \[\mc{A}(:,:,1)=\begin{bmatrix}
  1  &   3   &   1\\
	    -1  &   -3   &  -3\\
	    -1   &  2   &   2
  \end{bmatrix},~\mc{A}(:,:,2)=\begin{bmatrix}
   0   &   2  &    2\\
	     0   &  -2  &   -2\\
	     0    &  2   &   2
  \end{bmatrix},~\mc{A}(:,:,3)=\begin{bmatrix}
  0  &    4   &   0\\
	    -1  &   -3  &   -3\\
	     0    &  1     & 3
  \end{bmatrix},\]
  and $M=\begin{bmatrix}
 1     & 0   &  -1\\
     0    &  1   &   0\\
     0    &  1  &   -1
\end{bmatrix}$. Since $[\tilde{\mc{A}}(:,:,1)]^2=\begin{bmatrix}
    0& 0 &0\\
     0& 0 &0\\
     0& 0 &0
\end{bmatrix}=[\tilde{\mc{A}}(:,:,2)]^2$ and $[\tilde{\mc{A}}(:,:,3)]^3=\begin{bmatrix}
    0& 0 &0\\
     0& 0 &0\\
     0& 0 &0
\end{bmatrix}$, the tensor $\mc{A}$ is nilpotent with index $3=\max\{2,2,3\}$.
Using Algorithm \ref{alg:mdrazin}, we can evaluate the tubal index $k=3$ and $\mc{A}^\D=\mc{O}$.
\end{example}

Remark \ref{rem-nil} is stated from the Example \ref{ex-nil} and Algorithm \ref{alg:mdrazin}.

\begin{remark}\label{rem-nil}
    If tensor $\mc{A}\in\mbc^{m\times m\times p}$ is nilpotent then $\mc{A}^\D=\mc{O}$. Further, the equality $(\mc{A}^\D)^\D=A$ need not be true always.
\end{remark}

The  results below provide an affirmative answer to the second part of Remark \ref{rem-nil}.

\begin{theorem}
Let $M\in\mbc^{p\times p}$ and $\mc{A} \in \mbc^{m\times m\times p}$ be with the tubal index $k$.
Then $(\mc{A}^{\D})^{\D}=\mc{A}\Longleftrightarrow k=1$.
\end{theorem}

\begin{proof}
Let $k=1$. Then $\mc{A}=\mc{A}^\D\m\mc{A}^2=\mc{A}\m\mc{A}^\D\m\mc{A}$ and
$\mc{A}\m(\mc{A}^\D)^2=\mc{A}^\D$ and consequently, it follows $(\mc{A}^\D)^\D=\mc{A}$.

Conversely, let $(\mc{A}^{\D})^{\D}=\mc{A}$ and $\tilde{\mc{A}}=\mc{A}\times_3M$.
Then $\mat(\mc{A})=\mat((\mc{A}^{\D })^{\D })=([\mat(\mc{A})]^D)^D$.
Thus $\ind(\mat(\mc{A}))=1$ and subsequently, $\mathcal{R}(\mat(\mc{A}))=\mathcal{R}(\mat(\mc{A})^2)$, where $\mathcal{R}$ denotes the matrix range.
That leads to
\[\mathcal{R}(\blkd(\tilde{\mc{A}}(:,:,1),\tilde{\mc{A}}(:,:,2),\ldots,\tilde{\mc{A}}(:,:,p)))=\mathcal{R}(\blkd(\tilde{\mc{A}}^{2}(:,:,1),\tilde{\mc{A}}^2(:,:,2),\ldots,\tilde{\mc{A}}^2(:,:,p))).\]
Therefore, $\mathcal{R}(\tilde{\mc{A}}(:,:,i))=\mathcal{R}(\tilde{\mc{A}}(:,:,i)^2)$ for each $i=\overline{1,p}$ and we conclude $\ind(\tilde{\mc{A}}(:,:,i))=1$, $1\leq i\leq p$.
\end{proof}

In the next example, we illustrate that the reverse order law ($(\mc{A}\m\mc{B})^{\D}=\mc{B}^{\D }\m\mc{A}^{\D}$)
does not generally need to be true under the $M$-product.
\begin{example}\label{exa-rev}
Choose $\mc{A}$, $M$ as in Example \ref{exa3.4} and $\mc{B}\in \mbc^{m\times m\times p}$ with entries
\[\mc{B}(:,:,1)=\begin{bmatrix}
  2  & -18  &  -5\\
	    -1  &   9 &  -11\\
	     4  &  -4  &   2
\end{bmatrix},~
 \mc{B}(:,:,2)=\begin{bmatrix}
  0  &  10 &    3\\
	     0  &  -4  &   6\\
	    -2    & 2  &  -1
\end{bmatrix},~\mc{B}(:,:,3)=\begin{bmatrix}
  -2 &    6   &  1\\
	     1 &   -4 &    4\\
	    -2   &  2&    -1
\end{bmatrix}.\]
 Then we find that tubal index of $\mc{B}$ and $\mc{A}\m\mc{B}$, respectively 2 and 2. Further, we calculate
\[((\mc{A}\m\mc{B})^{\D})(:,:,1)-(\mc{B}^\D\m\mc{A}^\D)(:,:,1)=\begin{bmatrix}
 -0.4222&    0.7649 &  -1.6111\\
	    0.7000 &   0.1240 &  -2.7500\\
	    1.9111 &  -0.5684 &  -0.4444
\end{bmatrix}\neq \mc{O},\]
\[((\mc{A}\m\mc{B})^{\D})(:,:,2)-(\mc{B}^\D\m\mc{A}^\D)(:,:,2)=\begin{bmatrix}
0.4611 &   -0.3824  &   1.0556\\
	    0.1500 &   -0.0620  &   1.7500\\
	   -1.2056  &   0.2842  &   0.2222
\end{bmatrix}\neq \mc{O},\]
\[((\mc{A}\m\mc{B})^{\D})(:,:,3)-(\mc{B}^\D\m\mc{A}^\D)(:,:,3)=\begin{bmatrix}
 -0.0389  &  -0.3824 &    0.5556\\
	   -0.8500&    -0.0620 &    1.0000\\
	   -0.7056 &    0.2842  &   0.2222
\end{bmatrix}\neq \mc{O}.\]
\end{example}

If the tensors $\mc{A}$ and $\mc{B}$ are commutative, then the reverse order is true as given in Theorem \ref{ThmROLT1}.
\begin{theorem}\label{ThmROLT1}
Let $M\in\mbc^{p\times p}$ and $\mc{A}, ~\mc{B} \in \mbc^{m\times m\times p}$ with respective tubal indices $k_1$ and $k_2$.
If $\mc{A}\m\mc{B}=\mc{B}\m\mc{A}$, then  $(\mc{A}\m\mc{B})^{\D}=\mc{B}^{\D}\m\mc{A}^{\D}=\mc{A}^{\D}\m\mc{B}^{\D}$.
\end{theorem}

\begin{proof}
The assumption $\mc{A}\m\mc{B}=\mc{B}\m\mc{A}$ implies
\begin{eqnarray*}
\mc{A}^{\D }\m\mc{B}&=&{ (\mc{A}^{\D })^2\m\mc{A}\m\mc{B}=\mc{A}^{\D }\m(\mc{A}^\D \m\mc{B})\m\mc{A}}=(\mc{A}^{\D })^{k_1}\m(\mc{A}^\D \m\mc{B})\m\mc{A}^{k_1}\\
&=&(\mc{A}^{\D })^{k_1}\m(\mc{A}^\D \m\mc{B})\m\mc{A}^{k_1+1}\m\mc{A}^\D =(\mc{A}^{\D })^{k_1}\m\mc{A}^{k_1}\m\mc{B}\m\mc{A}^\D \\
&=&\mc{A}^{\D }\m\mc{A}\m\mc{B}\m\mc{A}^\D .
\end{eqnarray*}
Similarly, we verify that $\mc{B}\m\mc{A}^\D =\mc{A}^{\D }\m\mc{A}\m\mc{B}\m\mc{A}^{\D }=\mc{A}^\D \m\mc{B}$.
Next, $\mc{A}^\D \m\mc{B}^\D =\mc{B}^\D \m\mc{A}^\D $ follows from the identities listed below.
\begin{eqnarray*}
\mc{B}^{\D }\m\mc{A}^{\D }&=&\mc{B}^{\D }\m\mc{A}\m(\mc{A}^{\D })^2=\mc{A}\m\mc{B}^{\D }\m(\mc{A}^{\D })^{2}=\cdots=\mc{A}^{k_1}\m\mc{B}^{\D }\m(\mc{A}^{\D })^{k_1+1}\\
&=&\mc{A}^{k_1+1}\m\mc{A}^{\D }\m\mc{B}^{\D }\m(\mc{A}^{\D })^{k_1+1}=\mc{A}^{\D }\m\mc{A}^{k_1+1}\m\mc{B}^{\D }\m(\mc{A}^{\D })^{k_1+1}\\
&=&\mc{A}^{\D }\m\mc{B}^{\D }\m\mc{A}^{k_1+1}\m(\mc{A}^{\D })^{k_1+1}=\mc{A}^{\D }\m\mc{B}^{\D }\m\mc{A}\m\mc{A}^{\D },
\end{eqnarray*}
\begin{eqnarray*}
\mc{A}^{\D }\m\mc{B}^{\D } &=&(\mc{A}^{\D })^2\m\mc{A}\m\mc{B}^{\D }=(\mc{A}^{\D })^{2}\m\mc{B}^{\D }\m\mc{A}=\cdots=(\mc{A}^{\D })^{k_1+1}\m\mc{B}^{\D }\m\mc{A}^{k_1}\\
&=&(\mc{A}^{\D })^{k_1+1}\m\mc{B}^{\D }\m\mc{A}^{k_1+1}\m\mc{A}^{\D }=(\mc{A}^{\D })^{k_1+1}\m\mc{A}^{k_1}\m\mc{B}^{\D }\m\mc{A}\m\mc{A}^{\D }\\
&=&\mc{A}^{\D }\m\mc{B}^{\D }\m\mc{A}\m\mc{A}^{\D }.
\end{eqnarray*}
Let $\mc{X}=\mc{B}^{\D }\m\mc{A}^{\D }$ and $k=\max\{k_1,k_2\}$. Then
\begin{eqnarray*}
(\mc{A}\m\mc{B})^{k+1}\m\mc{X} &=& \mc{A}^{k+1}\m\mc{B}^{k+1}\m\mc{B}^{\D }\m\mc{A}^{\D }=\mc{A}^{k+1}{\star}_{M }\mc{B}^{k}\m\mc{A}^{\D }
=\mc{A}^k\m\mc{B}^k\\
&=&(\mc{A}\m\mc{B})^{k},
\end{eqnarray*}
\begin{eqnarray*}
\mc{X}\m(\mc{A}\m\mc{B})\m\mc{X} &=&\mc{B}^{\D }\m\mc{A}^{\D }\m\mc{A}\m\mc{B}\m\mc{B}^{\D }\m\mc{A}^{\D }\\
&=&\mc{B}^{\D }\m\mc{B}\m\mc{B}^\D \m\mc{A}^{\D }\m\mc{A}\m\mc{A}^{\D }=\mc{B}^{\D }\m\mc{A}^{\D }=\mc{X},
\end{eqnarray*}
\begin{eqnarray*}
\mc{A}\m\mc{B}\m\mc{X} & =&\mc{A}\m \mc{B}\m\mc{B}^{\D }\m\mc{A}^{\D }=\mc{A}\m\mc{B}^{\D }\m\mc{B}\m\mc{A}^{\D }=\mc{B}^{\D }\m\mc{A}\m\mc{A}^{\D }\m\mc{B}\\
&=&\mc{B}^{\D }\m\mc{A}^{\D }\m\mc{A}\m\mc{B}=\mc{X}\m\mc{A}\m\mc{B}.
\end{eqnarray*}
Hence, $\mc{X}=(\mc{A}\m\mc{B})^\D$.
As a result, $(\mc{A}\m\mc{B})^{\D }= \mc{B}^{\D }\m\mc{A}^{\D }=\mc{A}^{\D }\m\mc{B}^{\D }$ .
\end{proof}

In general, the Drazin inverse of a tensor sum is not equal to the sum of individual Drazin inverses. 
This is confirmed in Example \ref{ExmTD+}.

\begin{example}\label{ExmTD+}\rm
Consider $\mc{A}$, $M$ as in Example \ref{exa3.4} and $\mc{B}$ as in Example \ref{exa-rev}.
Thus the tubal index of both $\mc{A}$, $\mc{B}$ is $2$ and the tubal index of $\mc{A}\pm\mc{B}$ is $1$.
Next we evaluate,
\[(\mc{A} + \mc{B})^{\D }(:,:1)-(\mc{B}^\D +\mc{A}^\D )(:,:,1)=\begin{bmatrix}
4.372 &  29.160 & -31.112\\
	    0.378  & -1.160&   8.112\\
	  -10.628  & -0.340  & 15.888
\end{bmatrix}\neq\mc{O},\]
\[(\mc{A} + \mc{B})^{\D }(:,:2)-(\mc{B}^\D +\mc{A}^\D )(:,:,2)=\begin{bmatrix}
-2.436  &-16.580 &  17.056\\
	    0.186  &  0.580 &  -4.056\\
	    5.3140 &  0.170 &  -6.944
\end{bmatrix}\neq\mc{O},\]
\[(\mc{A} + \mc{B})^{\D }(:,:3)-(\mc{B}^\D +\mc{A}^\D )(:,:,3)=\begin{bmatrix}
 -1.936   &-8.580 &  11.056\\
	   -0.564   & 0.580 &  -3.056\\
	    5.314  &  0.170 &  -8.944
\end{bmatrix}\neq\mc{O},\]
\[(\mc{A} - \mc{B})^{\D }(:,:1)-(\mc{B}^\D -\mc{A}^\D )(:,:,1)=\begin{bmatrix}
19.372  & 25.910 &  -6.612\\
	   -3.622 & -12.160 &  25.112\\
	  -10.628  &  2.410  &  4.388
\end{bmatrix}\neq\mc{O},\]
\[(\mc{A} - \mc{B})^{\D }(:,:2)-(\mc{B}^\D -\mc{A}^\D )(:,:,2)=\begin{bmatrix}
 -10.936 & -14.955&    2.806\\
	    1.186  &  6.580 & -15.056\\
	    6.314 &  -1.205 &  -2.194
\end{bmatrix}\neq\mc{O},\]
\[(\mc{A} - \mc{B})^{\D }(:,:3)-(\mc{B}^\D -\mc{A}^\D )(:,:,3)=\begin{bmatrix}
 -6.436  & -6.955 &   2.806\\
	    2.436&   3.580 &  -7.056\\
	    4.314&  -1.205  & -2.194
\end{bmatrix}\neq\mc{O}.\]
 \end{example}
Results obtained in theorem \ref{ThmAddD} investigate the additive property of the Drazin inverse.

\begin{theorem}\label{ThmAddD}
Let $M\in\mbc^{p\times p}$ and $\mc{A}, ~\mc{B} \in \mbc^{m\times m\times p}$ with respective tubal index $k_1(\geq 1)$ and $k_2(\geq 1)$.
Then $\mc{A}\m\mc{B}=\mc{B}\m\mc{A}=\mc{O}$ implies $(\mc{A}\pm\mc{B})^{\D }=\mc{A}^{\D }\pm \mc{B}^{\D }$.
\end{theorem}
\begin{proof}
Let $\mc{A}\m\mc{B}=\mc{B}\m\mc{A}=\mc{O}$ and $k=\max\{k_1,k_2\}$. Then
 $(\mc{A}+\mc{B})^{k+1}=\mc{A}^{k+1}+\mc{B}^{k+1}$,
\[\mc{A}\m\mc{B}^{\D }=\mc{A}\m\mc{B}\m (\mc{B}^{\D })^2=\mc{O},\mbox{ and }\mc{B}\m\mc{A}^{\D }=\mc{O}.
\]
Now
\[(\mc{A}+\mc{B})^{k+1}\m(\mc{A}^{\D }+\mc{B}^{\D })=(\mc{A}^{k+1}+\mc{B}^{k+1})\m(\mc{A}^{\D }+\mc{B}^{\D })=\mc{A}^k+\mc{B}^k=(\mc{A}+\mc{B})^{k},\]
\[(\mc{A}^{\D }+\mc{B}^{\D })\m(\mc{A}+\mc{B})=(\mc{A}^{\D }\m\mc{A}+\mc{B}^{\D }\m\mc{B}) =  (\mc{A}^{\D }+\mc{B}^{\D })(\mc{A}+\mc{B}), \mbox{ and }\]
\begin{eqnarray*}
 (\mc{A}^{\D }+\mc{B}^{\D })\m(\mc{A}+\mc{B})\m(\mc{A}^{\D }+\mc{B}^{\D })&=&(\mc{A}^{\D }\m\mc{A}+\mc{B}^{\D }\m\mc{B})\m(\mc{A}+\mc{B})   \\
 &=&\mc{A}^{\D }\m\mc{A}\m\mc{A}^\D +\mc{B}^{\D }\m\mc{B}\m\mc{B}^\D =\mc{A}^{\D }+\mc{B}^{\D }.
\end{eqnarray*}
Thus $\mc{A}^{\D }+\mc{B}^{\D }=(\mc{A}+\mc{B})^\D$.
Similarly, $(\mc{A}-\mc{B})^{\D }=\mc{A}^{\D }-\mc{B}^{\D }$ can be verified.
\end{proof}

\section{Core-EP inverse of a tensor under $M$-product}

The core-EP inverse under $M$-product is established in Definition \ref{cepdef}.
\begin{definition}\label{cepdef}
Let $M\in\mbc^{p\times p}$ and $\mc{A} \in \mbc^{m\times m\times p}$ with  tubal index $k$. If a tensor $\mc{X}\in \mbc^{m\times m\times p}$ satisfies $\mc{X}\m\mc{A}^{k+1}=\mc{A}^k$, $\mc{A}\m\mc{X}^2=\mc{X}$ and $(\mc{A}\m\mc{X})^*=\mc{A}\m\mc{X}$ then $\mc{X}$ is called the core-EP inverse of $\mc{A}$ and denoted by $\mc{A}^{\ep}$.
\end{definition}
The core-EP inverse can be defined using $\mat$ and $\mat^{-1}$ operators, as stated in Proposition \ref{propep}.

\begin{proposition}\label{propep}
Let $M\in\mbc^{p\times p}$ and $\mc{A} \in \mbc^{m\times m\times p}$ with  $\ind(\mat(\mc{A})) = k$. Then
\[\mc{A}^{\ep} = \mat^{-1}(\mat(\mc{A})^{\ep}).\]
\end{proposition}
\begin{proof}
Let $\mc{X}=\mat^{-1}(\mat(\mc{A})^{\ep})$.
Thus using the properties of $\mat$ and $\mat^{-1}$, it can be concluded that
\begin{eqnarray*}
\mc{A}\m\mc{X}^2&=&\mat^{-1}(\mat(\mc{A}))\m\mat^{-1}((\mat(\mc{A})^{\ep})^2)\\
&=&\mat^{-1}(\mat(\mc{A})(\mat(\mc{A})^{\ep})^2) \\
&=&\mat^{-1}(\mat(\mc{A})^{\ep})=\mc{X},
    \end{eqnarray*}
    \begin{eqnarray*}
(\mc{A}\m\mc{X})^*&=&\mat^{-1}(\mat(\mc{A})(\mat(\mc{A})^{\ep})^*)=\mat^{-1}(\mat(\mc{A})\mat(\mc{A})^{\ep})\\
&=&\mat^{-1}(\mat(\mc{A}))\m\mat^{-1}(\mat(\mc{A})^{\ep})=\mc{A}\m\mc{X},
 \end{eqnarray*}
 \begin{eqnarray*}
\mc{X}\m\mc{A}^{k+1}&=&\mat^{-1}(\mat(\mc{A})^{\ep})\m\mat^{-1}(\mat(\mc{A})^{k+1})\\
&=&\mat^{-1}(\mat(\mc{A}^{\ep})(\mat(\mc{A})^{k+1}) \\
&=&\mat^{-1}(\mat(\mc{A})^k)=\mat^{-1}(\mat(\mc{A}^k))=\mc{A}^k.
    \end{eqnarray*}
Thus $\mc{A}^{\ep}=\mc{X}=\mat^{-1}(\mat(\mc{A})^{\ep})$.
\end{proof}

Algorithm \ref{alg:mep} is designed for computing the core-EP inverse in the framework of the $M$-product based on Definition \ref{cepdef}.
\begin{algorithm}[H]
 \caption{Core-EP inverse under $M$-product} \label{alg:mep}
\begin{algorithmic}[1]
  \Procedure{Core-EP inverse}{$\mc{A}^{\ep}$}
\State {\bf Input} $\mc{A} \in \mbc^{m \times m\times p}$ and  $M\in\mbc^{p\times p}$.
\State Compute $\tilde{\mc{A}}=\mc{A}\times_3 M$
\For{$i \gets 1$ to $p$}
\State $k_i=\ind(\tilde{\mc{A}}(:,:,i))$
   \EndFor
   \State Compute $k=\max_{1\leq i\leq p}k_i$
\For{$i \gets 1$ to $p$}
   \State $\mc{Z}(:,:,i) = (\tilde{\mc{A}}(:,:,i))^{\ep}$, where $A^{\ep}$ is the core-EP inverse of the matrix $A$ with index $k$
   \EndFor
\State Compute $\mc{X}=\mc{Z}\times_3M^{-1}$
\State \Return $\mc{A}^{\ep}=\mc{X}$
\EndProcedure
 \end{algorithmic}
\end{algorithm}

Based on Algorithm \ref{alg:mep}, in Example \ref{exa4.3} we compute the core-EP inverse.
\begin{example}\rm\label{exa4.3}
  Let $\mc{A}\in\mbc^{3\times 3\times 3}$ with entries
  \[\mc{A}(:,:,1)=\begin{bmatrix}
   2  &   2   &  -1\\
	    -2 &     0   &   0\\
	    -2  &    2  &   -1\\
  \end{bmatrix},~\mc{A}(:,:,2)=\begin{bmatrix}
  0  &   -2  &   -4\\
	    -8  &    7  &    0\\
	    11 &   -10    &  8
  \end{bmatrix},~\mc{A}(:,:,3)=\begin{bmatrix}
 -1  &   -1 &     3\\
	     1 &     2   &   1\\
	     0  &   -1   &   0
  \end{bmatrix}\]
  and $M=\begin{bmatrix}
 -1    &  0  &   -1\\
     1   &   0 &     0\\
    -1  &   -1  &   -1
\end{bmatrix}$.
Since $\ind(\tilde{\mc{A}}(:,:,1))=1,~\ind(\tilde{\mc{A}}(:,:,2))=2$ and $\ind(\tilde{\mc{A}}(:,:,3))=3$, the tubal index of $\mc{A}$ is equal to $k=3=\max\{1,2,3\}$.
Based on Algorithm \ref{alg:mep}, we calculate $\mc{X}=\mc{A}^{\ep}$, where
$$
\aligned
\mc{X}(:,:,1)&=\begin{bmatrix}
 0.0714  & -0.1429  & -0.2143\\
	   -0.1429  &  0.2857&    0.4286\\
	   -0.2143 &   0.4286 &   0.6429
\end{bmatrix},~\mc{X}(:,:,2)=\begin{bmatrix}
 -0.3158   & 0.2925 &   0.6525\\
	   -0.3417 &   0.0474  &  0.2991\\
	   -0.0620 &  -0.2299&   -0.2230
\end{bmatrix}\\
\mc{X}(:,:,3)&=\begin{bmatrix}
  0.2619 &  -0.1905  & -0.4524\\
	    0.4762  & -0.2857 &  -0.7619\\
	    0.2143  & -0.0952 &  -0.3095
     \end{bmatrix}.
     \endaligned
     $$
\end{example}
We now discuss the comparative analysis of the mean CPU time (MT$^M$) and errors (Error$^M$) generated while calculating the core-EP inverse for different choices of $M$.
The tensor $\mc{A}$ considered in Table \ref{tab:cep-m} is generated randomly such that its index is $1$ and $2$.

\begin{table}[H]
    \begin{center}
       \caption{Computational time for computing $\mc{A}^{\ep}$ for different tensor products}
        \vspace{0.2cm}
        \renewcommand{\arraystretch}{1.1}
        \small
        \begin{tabular}{cccccccc}
    \hline
        Size of $\mc{A}$& $k$  & MT$^{t}$  & MT$^{c}$ & MT$^M$ & Error$^{t}$ & Error$^{c}$ &Error$^M$\\
           \hline
    \multirow{4}{*}{$300\times 300\times 300$}& \multirow{4}{*}{1} &  \multirow{4}{*}{26.26} &\multirow{4}{*}{15.14}&\multirow{4}{*}{15.06}&
    $\mc{E}^{t}_{1^{1}}= 1.22e^{-08}$  & $\mc{E}^{c}_{1^{1}}=2.94e^{-08}$ & $\mc{E}^{M}_{1^{1}}=3.53e^{-07}$
    \\
& & & & & $\mc{E}^{t}_{3}= 1.67e^{-10}$  & $\mc{E}^{c}_{3}=8.18e^{-08}$ & $\mc{E}^{M}_{3}=1.66e^{-07}$
    \\
          &  & & & & $\mc{E}^{t}_{7}= 5.00e^{-09}$  & $\mc{E}^{c}_{7}=2.06e^{-06}$ & $\mc{E}^{M}_{7}=3.93e^{-07}$
      \\
           \hline
    \multirow{4}{*}{$400\times 400\times 400$}& \multirow{4}{*}{1} &  \multirow{4}{*}{58.75} &\multirow{4}{*}{39.32}&\multirow{4}{*}{39.11}&
    $\mc{E}^{t}_{1^{1}}= 3.01e^{-07}$  & $\mc{E}^{c}_{1^{1}}=2.29e^{-08}$ & $\mc{E}^{M}_{1^{1}}=2.62e^{-07}$
    \\
& & & & & $\mc{E}^{t}_{3}= 1.38e^{-09}$  & $\mc{E}^{c}_{3}=1.62e^{-08}$ & $\mc{E}^{M}_{3}=4.35e^{-08}$
    \\
          &  & & & & $\mc{E}^{t}_{7}= 9.64e^{-08}$  & $\mc{E}^{c}_{7}=1.41e^{-06}$ & $\mc{E}^{M}_{7}=2.80e^{-06}$
       \\
           \hline
    \multirow{4}{*}{$300\times 300\times 300$}& \multirow{4}{*}{2} &  \multirow{4}{*}{25.36} &\multirow{4}{*}{18.74}&\multirow{4}{*}{18.69}&
    $\mc{E}^{t}_{1^{2}}= 1.09e^{-07}$  & $\mc{E}^{c}_{1^{2}}=1.23e^{-07}$ & $\mc{E}^{M}_{1^{2}}=3.91e^{-07}$
    \\
& & & & & $\mc{E}^{t}_{3}= 3.90e^{-08}$  & $\mc{E}^{c}_{3}=3.45e^{-08}$ & $\mc{E}^{M}_{3}=5.75e^{-07}$
    \\
          &  & & & & $\mc{E}^{t}_{7}= 1.36e^{-09}$  & $\mc{E}^{c}_{7}=1.54e^{-09}$ & $\mc{E}^{M}_{7}=4.00e^{-09}$
      \\
           \hline
    \multirow{4}{*}{$400\times 400\times 400$}& \multirow{4}{*}{2} &  \multirow{4}{*}{58.78} &\multirow{4}{*}{47.55}&\multirow{4}{*}{46.09}&
    $\mc{E}^{t}_{1^{2}}= 1.09e^{-07}$  & $\mc{E}^{c}_{1^{2}}=1.23e^{-07}$ & $\mc{E}^{M}_{1^{2}}=3.91e^{-07}$
    \\
& & & & & $\mc{E}^{t}_{3}= 3.90e^{-08}$  & $\mc{E}^{c}_{3}=3.45e^{-08}$ & $\mc{E}^{M}_{3}=5.75e^{-07}$
    \\
          &  & & & & $\mc{E}^{t}_{7}= 1.36e^{-09}$  & $\mc{E}^{c}_{7}=1.54e^{-09}$ & $\mc{E}^{M}_{7}=4.00e^{-09}$
       \\
           \hline
    \end{tabular}
    \label{tab:cep-m}
    \end{center}
    \end{table}
\begin{example}\label{exa-ep-com}
Consider  $M=\mathrm{rand}(n)$ and $\mc{A} \in \mbc^{n \times n\times n}$ chosen using rule $\tilde{\mc{A}}(:,:,i)=gallery(\textbf{gearmat},n)$ for $i=\overline{1,n}$.
In this case, the tubal index of $\mc{A}$ is $2$. For tubal index $1$, we consider $\tilde{\mc{A}}(:,:,i)=\textbf{magic}(n)$.
A comparison of the mean CPU time (MT) for the randomly generated invertible matrix $M$, using the tubal index and $\mat$  operation is provided in Table \ref{tab:comp-ep}
\begin{table}[H]
    \begin{center}
    \caption{Comparison of mean CPU time for computing $\mc{A}^{\ep}$}
     \vspace{0.1cm}
         \renewcommand{\arraystretch}{1.1}
    \begin{tabular}{|ccc|ccc|}
    \hline
      \scriptsize{Size of $\mc{A}$}   & $k$ & \scriptsize{MT (Using tubal index)} & \scriptsize{Size of $\mat(\mc{A})$} & \scriptsize{$ind(\mat(\mc{A}))$} &\scriptsize{MT (Using $\mat$ and $\mat^{-1}$)} \\
         \hline
   \scriptsize{$60\times 60\times 60$}  & 1 &0.20 & \scriptsize{$3600\times 3600$}& 1 & 10.70  \\
      \hline
        \scriptsize{$80\times 80\times 80$} & 1 &0.39 &\scriptsize{$6400\times 6400$}& 1 & 47.58   \\
      \hline
         \scriptsize{$100\times 100\times 100$}  &2  &1.25 &\scriptsize{$10000\times 10000$}& 2&  217.78\\
      \hline
        \scriptsize{$120\times 120\times 120$}  &2  & 2.06  & \scriptsize{$14400\times 14400$} &2&  576.88 \\
      \hline
    \end{tabular}
        \label{tab:comp-ep}
        \end{center}
\end{table}
\end{example}

The results in Theorem \ref{ep-prel} are derived using the definition of the core-EP inverse.
\begin{theorem}\label{ep-prel}
Let $\mc{A} \in \mbc^{m\times m\times p}$ be tubal index $k$ with respect to $M\in\mbc^{p\times p}$.
Then $\mc{X}:=\mc{A}^{\ep}$ satisfies
\begin{enumerate}
\item[\rm (i)]$\mc{A}\m\mc{X}=\mc{A}^{n}\m\mc{X}^{n}$, $n\in\mathbb{N}$.
\item[\rm (ii)] $\mc{X}=\mc{A}^{n}\m\mc{X}^{n+1}$, $n\in\mathbb{N}$.
\item[\rm (iii)] $\mc{X}\m\mc{A}\m\mc{X}=\mc{X}$.
\item[\rm (iv)] $\mc{X}^{k+1}\m\mc{A}^{k}=\mc{A}^{\D }$.
\end{enumerate}
\end{theorem}
The core-EP inverse can be expressed in terms of the MP and Drazin inverses, as given in Theorem \ref{ThmTCEPD}.
\begin{theorem}\label{ThmTCEPD}
 Let $M\in\mbc^{p\times p}$ and $\mc{A} \in \mbc^{m\times m\times p}$ with  tubal index $k$. Then for $l\geq k$, \[\mc{A}^{\ep}=\mc{A}^{\D }\m\mc{A}^{l}\m(\mc{A}^{l})^{\dagger}.\]
\end{theorem}
\begin{proof}
   Let $\mc{X}=\mc{A}^{\D }\m\mc{A}^{l}\m(\mc{A}^{l})^{\dagger}$. Then
   \[\mc{X}\m\mc{A}^{l+1}= \mc{A}^{\D }\m\mc{A}^{l}\m(\mc{A}^{l})^{\dagger}\m\mc{A}^{l+1}=\mc{A}^\D \m\mc{A}^{l+1}=\mc{A}^l,\]
   \begin{eqnarray*}
   \mc{A}\m\mc{X}^2&=& \mc{A}\m\mc{A}^{\D }\m\mc{A}^{l}\m(\mc{A}^{l})^{\dagger}\m\mc{X}=\mc{A}^{l}\m(\mc{A}^{l})^{\dagger}\m\mc{A}^{\D }\m\mc{A}^{l}\m(\mc{A}^{l})^{\dagger}\\
   &=&\mc{A}^{\D }\m\mc{A}^{l}\m(\mc{A}^{l})^{\dagger}=\mc{X},
   \end{eqnarray*}
   and $(\mc{A}\m\mc{X})^*=(\mc{A}^{l}\m(\mc{A}^{l})^{\dagger})^*=\mc{A}^{l}\m(\mc{A}^{l})^{\dagger}=\mc{A}\m\mc{X}$. Thus, the proof is complete.
   \end{proof}

The additive properties of the core-EP inverse are discussed in Theorem \ref{ThmTAdd}.

\begin{theorem}\label{ThmTAdd}
Let $M\in\mbc^{p\times p}$ and $\mc{A}, ~\mc{B} \in \mbc^{m\times m\times p}$ with the respective tubal index $k_1$ and $k_2$. If $\mc{A}\m\mc{B}=\mc{B}\m\mc{A}=\mc{O}=\mc{A}^*\m\mc{B}$, then $(\mc{A}\pm\mc{B})^{\ep}=\mc{A}^{\ep}\pm\mc{B}^{\ep}$.
\end{theorem}
\begin{proof}
 Let $\mc{A}\m\mc{B}=\mc{B}\m\mc{A}=\mc{O}=\mc{A}^*\m\mc{B}$.
 Then it can be obtained
\[\mc{A}\m\mc{B}^{\ep}=\mc{A}\m\mc{B}\m(\mc{B}^{\ep})^{2}=\mc{O},\]
 \[\mc{B}\m\mc{A}^{\ep}=\mc{B}\m\mc{A}\m(\mc{A}^{\ep})^{2}=\mc{O},\]
 \[\mc{B}^{\ep}\m\mc{A}=\mc{B}^{\ep}\m(\mc{B}^{\ep})^{\star}\m\mc{B}^{\star}\m\mc{A}=\mc{O},\]
\[\mc{A}^{\ep}\m\mc{B}=\mc{A}^{\ep}\m\mc{A}\m\mc{A}^{\ep}\m\mc{B}=\mc{A}^{\ep}\m(\mc{A}^{\ep})^*\m\mc{A}^*\m\mc{B}=\mc{O},\]
 \[\mc{A}^{\ep}\m\mc{B}^{\ep}=\mc{A}^{\ep}\m(\mc{A}^{\ep})^{\star}\m\mc{A}^{\star}\m\mc{B}\m(\mc{B}^{\ep})^{2}=\mc{O},\]
\[\mc{B}^{\ep}\m\mc{A}^{\ep}=\mc{B}^{\ep}\m(\mc{B}^{\ep})^*\m\mc{B}^*\m\mc{A}\m(\mc{A}^{\ep})^{2}=\mc{O}.\]
Assuming $k=\max\{k_1,k_2\}$, it can be obtained by Theorem \ref{ep-prel} (ii)
\[\mc{A}^l=\mc{A}^{\ep}\m\mc{A}^{k+1}=\mc{A}^k\m(\mc{A}^{\ep})^k\m\mc{A}^k,~\mc{B}^k=\mc{B}^k\m(\mc{B}^{\ep})^k\m\mc{B}^k.\]
 Now, using the above properties, one obtains
\[(\mc{A}^{\ep}+\mc{B}^{\ep})\m(\mc{A}+\mc{B})^{k+1}=(\mc{A}^{\ep}+\mc{B}^{\ep})\m(\mc{A}^{k+1}+ \mc{B}^{k+1})=\mc{A}^k+ \mc{B}^k=(\mc{A}+\mc{B})^k,\]
\[(\mc{A}+\mc{B})\m(\mc{A}^{\ep}+\mc{B}^{\ep})^2=(\mc{A}+\mc{B})\m((\mc{A}^{\ep})^2+(\mc{B}^{\ep})^2)=\mc{A}^{\ep}+\mc{B}^{\ep},\]
and
\begin{eqnarray*}
((\mc{A}+\mc{B})\m(\mc{A}^{\ep}\m\mc{B}^{\ep}))^*&=&(\mc{A}\m\mc{A}^{\ep})^*+(\mc{B}\m\mc{B}^{\ep})^*=\mc{A}\m\mc{A}^{\ep}+\mc{B}\m\mc{B}^{\ep}\\
&=&(\mc{A}+\mc{B})\m(\mc{A}^{\ep}\m\mc{B}^{\ep}).
 \end{eqnarray*}
 Other part can be verified in the same way.
\end{proof}

\section {Composite generalized inverses}\label{SecCompGIT}

\begin{definition}
Let $\mc{A} \in \mbc^{m\times m\times p}$ be with the tubal index $k$ with respect to $M\in\mbc^{p\times p}$. The composite generalized inverses are defined as per the representation given in the  Table \ref{tab:compodef}.
\begin{table}[H]
    \begin{center}
        \caption{Composite generalized inverse of $\mc{A}$}
        \vspace{0.2cm}
         \renewcommand{\arraystretch}{1.1}
    \begin{tabular}{|c|c|c|}
    \hline
     Inverse   & Notation & Definition\\
         \hline
       DMP inverse  & $\mc{A}^{\D ,\dagger}$ &$\mc{A}^{\D ,\dagger}=\mc{A}^{\D }\m\mc{A}\m\mc{A}^{\dagger}$\\
         \hline
           MPD inverse  & $\mc{A}^{\dagger,\D}$ &${ \mc{A}^{\dagger,\D}}=\mc{A}^{\dagger}\m\mc{A}\m\mc{A}^\D $\\
                 \hline
        CMP inverse  & $\mc{A}^{c,\dagger}$ &$\mc{A}^{c,\dagger}=\mc{A}^{\dagger}\m\mc{A}\m\mc{A}^{\D }\m\mc{A}\m\mc{A}^{\dagger}=\mc{A}^{\dagger}\m\mc{A}\m\mc{A}^{\D ,\dagger}$\\
            \hline
    \end{tabular}
     \label{tab:compodef}
     \end{center}
\end{table}
\end{definition}

\begin{theorem}
Let $M\in\mbc^{p\times p}$ and $\mc{A} \in \mbc^{m\times m\times p}$ be with tubal index $k$.
Then $\mc{A}^{\D ,\dagger}$ is a unique solution to subsequent tensor equations
\begin{equation}\label{dmpeq}
   \mc{Y}\m\mc{A}\m\mc{Y}=\mc{Y},\ ~\mc{Y}\m\mc{A}=\mc{A}^{\D }\m\mc{A},\ ~\mc{A}\m\mc{Y} = \mc{A}\m\mc{A}^\D \m\mc{A}\m \mc{A}^{\dagger}.
\end{equation}
\end{theorem}

\begin{proof}
It can be verified that $\mc{Y}:=\mc{A}^{\D ,\dagger}$ satisfies the equation \eqref{dmpeq}.
Next, we will verify the uniqueness. Suppose two solutions exist, denoted by $\mc{Y}$ and $\mc{Z}$.
Now
 \begin{eqnarray*}
\mc{Y} &=& \mc{Y}\m\mc{A}\m\mc{Y}=\mc{A}^{\D }\m\mc{A}\m\mc{Y}=\mc{Z}\m\mc{A}\m\mc{Y}\\
&=&\mc{Z}\m\mc{A}\m\mc{A}^\D \m\mc{A}\m \mc{A}^{\dagger}\\
&=&\mc{Z}\m\mc{A}\m\mc{Z}=\mc{Z}
   \end{eqnarray*}
   confirms equality between $\mc{Y}$ and $\mc{Z}$.
 \end{proof}

The following results for other composite generalized inverses are verified similarly.
\begin{theorem}
Let $M\in\mbc^{p\times p}$ and $\mc{A} \in \mbc^{m\times m\times p}$ with tubal index $k$.
Under these assumptions, $\mc{A}^{\dagger,\D}$ is a unique solution to tensor equations
\[ \mc{Y}\m\mc{A}\m\mc{Y}=\mc{Y},~\mc{Y}\m\mc{A}=\mc{A}^{\dagger}\m\mc{A}\m\mc{A}^\D \m \mc{A},~\mc{A}\m\mc{Y} = \mc{A}\m\mc{A}^\D .\]
\end{theorem}
\begin{theorem}
 Let $M\in\mbc^{p\times p}$ and $\mc{A} \in \mbc^{m\times m\times p}$ be with tubal index $k$.
 Then $\mc{A}^{c,\dagger}$ is a unique solution to
\[ \mc{Y}\m\mc{A}\m\mc{Y}=\mc{Y},\ ~\mc{Y}\m\mc{A}=\mc{A}^{\dagger,\D}\m\mc{A},\ ~\mc{A}\m\mc{Y} = \mc{A}\m\mc{A}^{\D ,\dagger}.\]
\end{theorem}
\begin{example}\rm
  Choose $\mc{A}$ and $M$ as in Example \ref{exa4.3}. We can evaluate
   $\mc{X}=\mc{A}^{\D ,\dagger}$, $\mc{Y}=\mc{A}^{\dagger,\D}$ and $\mc{Z}=\mc{A}^{c,\dagger}$, where
   \[\mc{X}(:,:,1)=\begin{bmatrix}
       0.8333 &  -0.3333 &   0.1667\\
	   -1.6667  &  0.6667 &  -0.3333\\
	   -2.5  &  1 &  -0.5
   \end{bmatrix},~\mc{X}(:,:,2)=\begin{bmatrix}
  -0.3158  &  0.2925  &  0.6525\\
	   -0.3417  &  0.0474   & 0.2991\\
	   -0.0620   &-0.2299  & -0.2230
   \end{bmatrix},\]
    \[\mc{X}(:,:,3)=\begin{bmatrix}
      -0.5 & 0 &  -0.8333\\
	    2  &-0.6667  &0\\
	    2.5  & -0.6667   & 0.8333
   \end{bmatrix},~\mc{Y}(:,:,1)=\begin{bmatrix}
 2   & 2   &-1\\
	   -0.8 &  -0.8  &0.4\\
	    0.4 &   0.4  &-0.2
   \end{bmatrix},\]
    \[\mc{Y}(:,:,2)=\begin{bmatrix}
      -0.2973  &   0.2973   &  0.0155\\
	   -0.3694    & 0.0361  &  -0.3488\\
	   -0.7595     &0.4261  &  -0.2302
   \end{bmatrix},~\mc{Y}(:,:,3)=\begin{bmatrix}
 -1.6667  & -2.3333 &   1\\
	    1.1333  &  0.8&   -0.0667\\
	    0.2667 &  -0.7333  &  0.5333
   \end{bmatrix},\]
    \[\mc{Z}(:,:,1)=\begin{bmatrix}
     0.8333 &  -0.3333  &  0.1667\\
	   -0.3333 &   0.1333  & -0.0667\\
	    0.1667  & -0.0667 &   0.0333
   \end{bmatrix},~\mc{Z}(:,:,2)=\begin{bmatrix}
 -0.0982  &  0.0670  &  0.2322\\
	   -0.1241  & -0.1781   &-0.1211\\
	   -0.2842  & -0.0076  &  0.2215
   \end{bmatrix}\]
   and $\mc{Z}(:,:,3)=\begin{bmatrix}
 -0.7222  &  0.2222 &  -0.3889\\
   0.4444  &  0.0889  &  0.1778\\
  0.0556  &  0.1778  & -0.1444
   \end{bmatrix}.$
\end{example}

\begin{theorem}
 Let $M\in\mbc^{p\times p}$ and $\mc{A} \in \mbc^{m\times m\times p}$ with the tubal index $k$.
 Then $\mc{A}^{c,\dagger}$ is the unique solution to the tensor equations
  \begin{equation}\label{eqcmp-rg}
  \mc{A}\m \mc{Y}\m\mc{A}=\mc{A}\m\mc{A}^{\D }\m\mc{A},~ \rg(\mc{Y})\subseteq \rg(\mc{A}^*),~\rg(\mc{Y}^*)\subseteq\rg(\mc{A}).
  \end{equation}
\end{theorem}
\begin{proof}
Let $\mc{Y}=\mc{A}^{c,\dagger}$. Then
\[\mc{A}\m \mc{Y}\m\mc{A}=\mc{A}\m \mc{A}^{\dagger}\m\mc{A}\m\mc{A}^\D \m\mc{A}\m\mc{A}^{\dagger}\m\mc{A}=\mc{A}\m\mc{A}^{\D }\m\mc{A}.\]
The range conditions follow from the following identities:
\[\mc{Y}=\mc{A}^{\dagger}\m\mc{A}\m\mc{A}^{\D ,\dagger}=(\mc{A}^{\dagger}\m\mc{A})^*\m\mc{A}^{\D ,\dagger}=\mc{A}^*\m(\mc{A}^{\dagger})^*\m\mc{A}^{\D ,\dagger},\]
\[\mc{Y}^*=(\mc{A}^{\dagger,\D}\m\mc{A}\m\mc{A}^{\dagger})^*=(\mc{A}^{\dagger,\D}\m\mc{A}\m(\mc{A}^{\dagger})^*)^*=\mc{A}\m(\mc{A}^{\dagger}\m(\mc{A}^{\dagger,\D})^*.\]
Suppose that there exist two solutions $\mc{Y}$ and $\mc{Z}$, that satisfy equation \eqref{eqcmp-rg}. Then
\begin{equation}\label{eq-pr1}
  \mc{A}\m \mc{Y}\m\mc{A}=\mc{A}\m\mc{A}^{\D }\m\mc{A},~  \mc{Y}=\mc{A}^{*}\m\mc{S}_1,~\mc{Y}^*=\mc{A}\m\mc{T}_1,
 \end{equation}
 \begin{equation}\label{eq-pr2}
  \mc{A}\m\mc{Z}\m\mc{A}= \mc{A}\m\mc{A}^{\D }\m\mc{A},~\mc{Z}= \mc{A}^{*} \m\mc{S}_2,~\mc{Z}^*=\mc{A}\m\mc{T}_2
\end{equation}
for some $\mc{S}_1,~\mc{S}_2,~\mc{T}_1,~\mc{T}_2\in\mbc^{m\times 1\times p}$. Define $\mc{X}=\mc{Y}-\mc{Z},~\mc{S}=\mc{S}_1-\mc{S}_2$ and $\mc{T}=\mc{T}_1-\mc{T}_2$.
Then by equations \eqref{eq-pr1} and \eqref{eq-pr2}, one obtains
\begin{equation}\label{eq-pr3}
\mc{A}\m\mc{X}\m\mc{A}=\mc{O},~\mc{X}=\mc{A}^*\m\mc{S},~\mc{X}^*=\mc{A}\m\mc{T}.
\end{equation}
An application of the equation \eqref{eq-pr3} implies
\[
\mc{A}\m\mc{X}\m(\mc{A}\m\mc{X})^*=\mc{A}\m\mc{X}\m\mc{X}^*\m\mc{A}^*=\mc{A}\m\mc{X}\m\mc{A}\m\mc{T}\m\mc{A}^*=\mc{O},\]
and hence $\mc{A}\m\mc{X}=\mc{O}$. Now $\mc{X}^*\m\mc{X}=\mc{S}^*\m\mc{A}\m\mc{X}=\mc{O}$.
Therefore, $\mc{X}=\mc{O}$ and completes the proof.
\end{proof}

The following results for the DMP and MPD inverses are derived using similar principles.
\begin{theorem}
Let $M\in\mbc^{p\times p}$ and $\mc{A} \in \mbc^{m\times m\times p}$ with tubal index $k$.
Then $\mc{A}^{\D ,\dagger}$ is a unique solution to the equations
   \[
   \mc{A}^{\dagger}\m\mc{Y}\m \mc{A} = \mc{A}^{\dagger,\D},~\rg(\mc{Y})\subseteq \rg(\mc{A}^*),~\rg(\mc{Y}^*)\subseteq
  \rg(\mc{A}).\]
\end{theorem}

\begin{theorem}
 Let $M\in\mbc^{p\times p}$ and $\mc{A} \in \mbc^{m\times m\times p}$ with tubal index $k$.
Under these assumptions $\mc{A}^{\dagger,\D}$ represents a unique solution to the following constrained tensor equation
   \[
   \mc{A}\m\mc{Y}\m \mc{A}^{\dagger} = \mc{A}^{\D ,\dagger},~\rg(\mc{Y})\subseteq \rg(\mc{A}^*),~\rg(\mc{Y}^*)\subseteq
  \rg(\mc{A}).\]
\end{theorem}

\section{Solution of Multilinear System}

the solutions of multilinear systems in terms of the Drazin inverse and core-EP inverse are the main objectives of this section.
The following singular tensor equation is considered based on the given $\mc{A}\in\mbc^{m\times m\times p }$:
\begin{equation}\label{eqn-multi}
\mc{A}\m\mc{X}=\mc{B},~~~
\mc{X},~\mc{B}\in\mbc^{m\times 1\times p}.
\end{equation}

\begin{proposition}\label{prop-sol}
Let $M\in\mbc^{p\times p}$ and $\mc{A}\in \mbc^{m\times m\times p}$ with tubal index $k$.
Then $\mc{A}^{\D }\m\mc{B}$ satisfies the equation \eqref{eqn-multi} if and only if  $\mc{B}\in \rg(\mc{A}^{k})$.
\end{proposition}

\begin{proof}
Let $\mc{B}\in \rg(\mc{A}^{k})$. Then $\mc{B}=\mc{A}^{k}\m\mc{T}$ for some tensor $\mc{T}\in\mbc^{m\times 1\times p}$.
Now
\[
\mc{A}\m(\mc{A}^{\D }\m\mc{B})=\mc{A}\m\mc{A}^{\D }\m\mc{A}^{k}\m\mc{T}=\mc{A}^{k}\m\mc{T}=\mc{B}.\]
Conversely, if $\mc{A}^{\D }\m\mc{B}$ is a solution of \eqref{eqn-multi}, then  \[\mc{B}=\mc{A}\m\mc{A}^{\D }\m\mc{B}=\mc{A}^2\m(\mc{A}^{\D })^2\m\mc{B}=\cdots=\mc{A}^k\m(\mc{A}^\D )^k\m\mc{B}.\]
Hence, from Proposition \ref{proprn9}, it follows $\mc{B}\in \rg(\mc{A}^{k})$, which completes the proof.
\end{proof}

\begin{theorem}
Let $M\in\mbc^{p\times p}$ and $\mc{A}\in \mbc^{m\times m\times p}$ be with tubal index $k$.
If  $\mc{B}\in \rg(\mc{A}^{k})$, then  $\mc{A}^{\D }\m\mc{B}$ is the only solution to $\mc{A}\m\mc{X}=\mc{B}$ in $\rg(\mc{A}^{k})$.
\end{theorem}
\begin{proof}
From Proposition \ref{prop-sol}, it follows that $\mc{X}=\mc{A}^\D \m\mc{B}$ fulfils $\mc{A}\m\mc{X}=\mc{B}$ and
\[\mc{X}=\mc{A}^\D \m\mc{B}=\mc{A}\m(\mc{A}^\D )^2\m\mc{B}=\cdots=\mc{A}^k\m(\mc{A}^\D )^{k+1}\m\mc{B}\in\rg(\mc{A}^k).\]
Next we claim the uniqueness of the solution.
Suppose $\mc{Y}$ is any other solution in  $\rg(\mc{A}^{k})$.
Then $\mc{Y}=\mc{A}^{k}\m\mc{T}$ for some $\mc{T}\in\mbc^{m\times 1\times p}$.
Application of this equality produces
\[\mc{Y}=\mc{A}^{k}\m\mc{Y}=\mc{A}^{\D }\m\mc{A}^{k+1}\m\mc{Y}=\mc{A}^{\D }\m\mc{A}\m\mc{Y}=\mc{A}^{\D }\m\mc{B}=\mc{X},
\]
which completes the proof.
\end{proof}
\begin{theorem}
Let $M\in\mbc^{p\times p}$ and $\mc{A}\in \mbc^{m\times m\times p}$ have tubal index $k$.
The general solution to
\begin{equation}\label{eq-norm}
    \mc{A}^{k+1}\m\mc{X}=\mc{A}^{k}\m\mc{B}
\end{equation}
is given by $\mc{X}= \mc{A}^{\D }\m\mc{B}+(\mc{I}-\mc{A}^\D \m\mc{A})\m\mc{Z}$, where $\mc{Z}\in\mbc^{m\times 1\times p}$ is arbitrary.
\end{theorem}

\begin{proof}
The tensor $\mc{X}= \mc{A}^{\D }\m\mc{B}+(\mc{I}-\mc{A}^\D \m\mc{A})\m\mc{Z}$ satisfies
\begin{eqnarray*}
 \mc{A}^{k+1}\m\mc{X}&=&\mc{A}^{k+1}\m\mc{A}^{\D }\m\mc{B}+\mc{A}^{k+1}\m(\mc{I}-\mc{A}^\D \m\mc{A})\m\mc{Z}\\
 &=&\mc{A}^k\m\mc{B}+(\mc{A}^{k+1}-\mc{A}^{k+1})\m\mc{Z}=\mc{A}^k\m\mc{B}.
\end{eqnarray*}
If $\mc{Y}$ is any other solution of \eqref{eq-norm}, then $\mc{Y}$ satisfies
\begin{eqnarray*}
    \mc{Y}&=&\mc{A}^\D \m\mc{B}+\mc{Y}-\mc{A}^\D \m\mc{B}=\mc{A}^\D \m\mc{B}+\mc{Y}-(\mc{A}^\D )^{k+1}\m\mc{A}^{k}\m\mc{B}\\
    &=&\mc{A}^\D \m\mc{B}+\mc{Y}-(\mc{A}^\D )^{k+1}\m\mc{A}^{k+1}\m\mc{Y}=\mc{A}^\D \m\mc{B}+\mc{Y}-\mc{A}^\D \m\mc{A}\m\mc{Y}.
\end{eqnarray*}
Thus $\mc{Y}$ is defined by the expression $\mc{A}^{\D }\m\mc{B}+(\mc{I}-\mc{A}^\D \m\mc{A})\m\mc{Z}$, which was our goal.
\end{proof}
\begin{theorem}
Let $M\in\mbc^{p\times p}$ and $\mc{A}\in \mbc^{m\times m\times p}$ be with the tubal index $k$.
Under the assumption $\mc{B}\in \rg(\mc{A}\m\mc{A}^{\ep})$, the general solution to \eqref{eq-norm} is given by
\[\mc{X}=\mc{A}^{\ep}\m\mc{B}+(\mc{I}-\mc{A}^{\D }\m\mc{A})\m\mc{Z},\]
where $\mc{Z}\in\mbc^{m\times 1\times p}$ be arbitrary.
\end{theorem}

\begin{proof}
Let  $\mc{B}\in \rg(\mc{A}\m\mc{A}^{\ep})$.
Then   $\mc{B}=\mc{A}\m\mc{A}^{\ep}\m\mc{T}$  for some $\mc{T}\in\mathbb{C}^{m\times 1\times p }$.
The expression $\mc{X}=\mc{A}^{\ep}\m\mc{B}+(\mc{I}-\mc{A}^{\D }\m\mc{A})\m\mc{Z}$ satisfies
\begin{eqnarray*}
    \mc{A}^{k+1}\m\mc{X}&=&\mc{A}^{k+1}\m\mc{A}^{\ep}\m\mc{B}+\mc{A}^{k+1}\m( \mc{I}-\mc{A}^{\D }\m\mc{A})\m\mc{Z}=\mc{A}^{k+1}\m\mc{A}^{\ep}\m\mc{B}\\
&=&\mc{A}^{k+1}\m\mc{A}^{\ep}\m\mc{A}\m\mc{A}^{\ep}\m\mc{T}=\mc{A}^k\m\mc{A}\m\mc{A}^{\ep}\m\mc{T}\\
&=&\mc{A}^k\m\mc{B}.
\end{eqnarray*}
If $\mc{Y}$ is any other solution of \eqref{eq-norm}, it can be expressed as
\begin{eqnarray*}
    \mc{Y}&=&\mc{A}^{\ep}\m\mc{B}+\mc{Y}-\mc{A}^{\ep}\m\mc{B}=\mc{A}^{\ep}\m\mc{B}+\mc{Y}-\mc{A}^{\ep}\m\mc{A}\m\mc{A}^{\ep}\m\mc{B}\\
    &=&\mc{A}^{\ep}\m\mc{B}+\mc{Y}-\mc{A}^{\ep}\m\mc{A}^k\m(\mc{A}^{\ep})^k\m\mc{B}\\
    &=& \mc{A}^{\ep}\m\mc{B}+\mc{Y}-\mc{A}^{\ep}\m\mc{A}^{k+1}\m\mc{A}^\D \m(\mc{A}^{\ep})^k\m\mc{B}\\
     &=& \mc{A}^{\ep}\m\mc{B}+\mc{Y}-\mc{A}^\D \m\mc{A}\m\mc{A}^{\ep}\m\mc{B}\\
     &=&\mc{A}^{\ep}\m\mc{B}+\mc{Y}-\mc{A}^\D \m\mc{A}\m\mc{A}^{\ep}\m\mc{A}\m\mc{A}^{\ep}\m\mc{T}\\
    &=&\mc{A}^\D \m\mc{B}+\mc{Y}-\mc{A}^\D \m\mc{B}=\mc{A}^\D \m\mc{B}+\mc{Y}-(\mc{A}^\D )^k\m\mc{A}^k\m\mc{B}\\
    &=&\mc{A}^\D \m\mc{B}+\mc{Y}-(\mc{A}^\D )^k\m\mc{A}^{k+1}\m\mc{Y}=\mc{A}^\D \m\mc{B}+(\mc{I}-\mc{A}^{\D }\m\mc{A})\m\mc{Y}.
\end{eqnarray*}
\end{proof}

\begin{theorem}
   Let $\mc{A}\in \mbc^{m\times m\times p}$  be with the tubal index $k$. 
   The general solution to
   \begin{equation}\label{Equ35}
   \mc{A}^{k}\m \mc{Y} =\mc{A}^{k}\m  \mc{A}^{\dagger}\m  \mc{B}
   \end{equation}
is represented by
   $$\mc{Y}=\mc{A}^{c,\dagger}\m  \mc{B}+(\mc{I}-\mc{A}^{c,\dagger}\m \mc{A})\m \mc{Z}$$
 \end{theorem}
 \begin{proof}
 Suppose $\mc{Y}=\mc{A}^{c,\dagger} \m \mc{B}+(\mc{I}-\mc{A}^{c,\dagger}\m \mc{A})\m \mc{Z}$. Then\\

 $\mc{A}^{k}\m \mc{Y}=\mc{A}^{k}\m \mc{A}^{c,\dagger}\m \mc{B}+(\mc{A}^{k}-\mc{A}^{k}\m \mc{A}^{c,\dagger}\m \mc{A})\m\mc{Z}=\mc{A}^{k}\m \mc{A}^{c,\dagger}\m \mc{B}=\mc{A}^{k}\m \mc{A}^{\dagger}\m \mc{B}.$\\
Hence, $\mc{Y}$ satisfies the equation \eqref{Equ35}.

If $\mc{Y}_1$ is any other solution of equation \eqref{Equ35}, then it can be concluded that
\begin{align*}
\mc{Y}_1&=\mc{A}^{c,\dagger}\m \mc{A}\m \mc{Y}_1+\mc{Y}_1-\mc{A}^{c,\dagger}\m \mc{A}\m \mc{Y}_1\\
&=\mc{A}^{\dagger}\m \mc{A}\m \mc{A}^{\D }\m \mc{A}\m \mc{Y}_1+\mc{Y}_1-\mc{A}^{c,\dagger}\m \mc{A}\m \mc{Y}_1\\
&=\mc{A}^{\dagger}\m \mc{A}\m (\mc{A}^{\D })^{k}\m\mc{A}^{k}\m \mc{Y}_1+\mc{Y}_1 - \mc{A}^{c,\dagger}\m \mc{A}\m \mc{Y}_1\\
&=\mc{A}^{\dagger}\m \mc{A}\m \mc{A}^{\D }\m \mc{A}\m \mc{A}^{\dagger}\m \mc{B})+\mc{Y}_1-\mc{A}^{c,\dagger}\m \mc{A}\m \mc{Y}_1\\
&=\mc{A}^{c,\dagger}\m \mc{B}+\mc{Y}_1-\mc{A}^{c,\dagger}\m \mc{A})\m \mc{Y}_1\\
&=\mc{A}^{c,\dagger}\m \mc{B}+(\mc{I}-\mc{A}^{c,\dagger}\m \mc{A})\m \mc{Y}_1.
\end{align*}
Hence,  this solution is unique to this system.
 \end{proof}

\begin{theorem}\label{thmcmp}
  Let $M\in\mbc^{p\times p}$ and $\mc{A}\in \mbc^{m\times m\times p}$ with a tubal index $k$.
Then the assumption $\mc{B}\in \rg(\mc{A}^{k})$ implies that $\mc{A}^{c,\dagger}\m\mc{B}$ is the unique solution of $\mc{A}\m\mc{X}=\mc{B}$ in $\rg(\mc{A}^{\dagger}\m\mc{A}^{k})$.
 \end{theorem}
 \begin{proof}
Clearly $\mc{A}^{c,\dagger}\m\mc{B}\in \rg(\mc{A}^{\dagger}\m\mc{A}^{k})$ because $\mc{A}^{c,\dagger}\m\mc{B}=\mc{A}^{\dagger}\m\mc{A}^k\m(\mc{A}^\D )^k\m\mc{A}\m\mc{A}^{\dagger}\m\mc{B}$. Let $\mc{B}\in\rg(\mc{A}^k)$. Then $\mc{B}=\mc{A}^k\m\mc{T}$ for some $\mc{T}\in\mbc^{m\times 1\times p}$. Now
\begin{eqnarray*}
\mc{A}\m\mc{A}^{c,\dagger}\m\mc{B}&=&\mc{A}\m\mc{A}^{\dagger}\m\mc{A}\m\mc{A}^\D \m\mc{A}\m\mc{A}^{\dagger}\m\mc{B}\\
&=&\mc{A}\m\mc{A}^\D \m\mc{A}\m\mc{A}^{\dagger}\m\mc{A}^k\m\mc{T}=\mc{A}\m\mc{A}^\D \m\mc{A}^k\m\mc{T}\\
&=&\mc{A}^k\m\mc{T}=\mc{B}.
\end{eqnarray*}
If $\mc{Y}$ and $\mc{Z}$ are two solutions in $\rg(\mc{A}^{\dagger}\m\mc{A})^{k})$, then
\begin{eqnarray*}
\mc{Y}-\mc{Z}&\in& \nl(\mc{A})\cap\rg(\mc{A}^{\dagger}\m\mc{A}^{k})\subseteq \nl(\mc{A}^{\dagger}\m\mc{A})\cap\rg(\mc{A}^{\dagger}\m\mc{A}^{k})\\
&&\subseteq \nl(\mc{A}^{\dagger}\m\mc{A})\cap\rg(\mc{A}^{\dagger}\m\mc{A})=\{\mc{O}\}.
\end{eqnarray*}
This  verified the uniqueness of the solution.
 \end{proof}
 The following two results follow by using lines similar to Theorem  \ref{thmcmp}.

\begin{theorem}
 Let $M\in\mbc^{p\times p}$ and $\mc{A}\in \mbc^{m\times m\times p}$ with a tubal index $k$.
 If  $\mc{B}\in \rg(\mc{A}^{k})$, then  $\mc{A}^{\D,\dagger}\m\mc{B}$ is a unique solution to $\mc{A}\m\mc{X}=\mc{B}$ in $\rg(\mc{A}^\D \m\mc{A}^{k})$.
\end{theorem}
\begin{corollary}
Let $M\in\mbc^{p\times p}$ and $\mc{A}\in \mbc^{m\times m\times p}$ with tubal index $k$.
If  $\mc{B}\in \rg(\mc{A}^{k})$, then  $\mc{A}^{\dagger,\D}\m\mc{B}$ represents a unique solution of $\mc{A}\m\mc{X}=\mc{B}$ in $\rg(\mc{A}^{\dagger}\m\mc{A}^{k})$.
\end{corollary}

\subsection{Iterative methods}

Multilinear systems have been encountered in several fields of practical importance.
Following this, there is a growing interest in the development of tensor-based iterative methods for solving multilinear systems with the support of the Einstein product \cite{AsishRJ,Brazell} and the t-product \cite{kilmer13}.
In this section, we solve specific multilinear systems associated with invertible tensors using the Jacobi and Gauss-Seidel methods.

It is well known that the iterative method for solving the tensor equation $\mc{A}\m \mc{X} =\mc{B}$ starts with an initial approximation $\mc{X}^{(0)}$ and generates a sequence of solutions $\{\mc{X}\}_{s=1}^{\infty}$ that converges to $\mc{X}$.
A generic pattern of the iterative method in the tensor domain is defined as
\begin{equation}\label{ITS}
\mc{X}^{s}=\mc{T}\m\mc{X}^{s-1}+\mc{C} ~~\text{for}~~ s=1,2,\ldots
\end{equation}
where $\mc{T}$ denotes the iteration tensor.
The approximate solution $\mc{X}^{s}$ converges to the exact solution $\mc{X}=\mc{A}^{-1}\m\mc{B}$ as $s\rightarrow \infty$.

\begin{algorithm}[H]
\caption{Higher order Jacobi Method based on $M$-product} \label{alg:mjacobi}
\begin{algorithmic}[1]
\Procedure{Jacobi}{$\mc{A},\mc{B},\epsilon, \mbox{MAX}$}
\State {\bf Input} $\mc{A} \in \mbc^{m \times m\times p}, \mc{B} \in \mbc^{m \times 1\times p}$ and  $M\in\mbc^{p\times p}$.
\State Compute $\tilde{\mc{A}}=\mc{A}\times_3 M$
\For{$i=1$ to $p$}
\State Compute $\tilde{\mc{D}}(:,:,i)=diag(\tilde{\mc{A}}(:,:,i)),~\tilde{\mc{F}}(:.:,i)=\tilde{\mc{A}}(:,:,i)-\tilde{\mc{D}}(:,:,i)$
\State Compute $\tilde{\mc{T}}(:,:,i)=-(\tilde{\mc{D}})^{-1}(:,:,i)\tilde{\mc{F}}(:.:,i)$ and
$\mc{C}(:,1,i)=(\tilde{\mc{D}})^{-1}(:,:,i)\mc{B}(:,1,i)$
\State Initial guess $\tilde{\mc{X}}^{0}(:,1,i)$
\For{$s=1$ to MAX}
\State $\tilde{\mc{X}}^{s}(:,1,i)=\tilde{\mc{T}}(:,:,i)\tilde{\mc{X}}^{s-1}(:,1,i)+\mc{C}(:,1,i)$
\If
{$\|\tilde{\mc{X}}^{s}(:,1,i)-\tilde{\mc{X}}^{0}(:,1,i)\|\leq \epsilon$}
\State \textbf{break}
\EndIf
\State $\tilde{\mc{X}}^{0}(:,1,i)\leftarrow \tilde{\mc{X}}^{s}(:,1,i)$
\EndFor
\EndFor
 \State Compute $\mc{X}^{s}=\tilde{\mc{X}}^{s}\times_3M^{-1}$
\State \Return $\mc{X}^{s}$
\EndProcedure
 \end{algorithmic}
\end{algorithm}

\begin{theorem}\label{th69}
Let $\mc{A}\in\mathbb{C}^{m \times m \times p}$ and $M\in\mathbb{C}^{p \times p}$. Then
    \begin{itemize}
\item[\rm (i)]  $||\mc{A}^2||_{M} \leq ||\mc{A}||_{M}^2$.
\item[\rm (ii)] $\displaystyle{\lim_{s \rightarrow \infty} \mc{A}^{s}=\mc{O}}$ if  $\| \mc{A} \|_M < 1 $ or if and only  $\rho (\mc{A}) < 1$.
\item[\rm (iii)] The series $\displaystyle{\sum_{s=0}^{\infty}} \mc{A}^s$ converges if and only if $\displaystyle{\lim_{s\to\infty} \mc{A}^s=\mc{O}}$.
In addison, the series converges to $(\mc{I}-\mc{A})^{-1}$.
    \end{itemize}
\end{theorem}

\begin{proof}
\noindent (i)~Following the tubal norm definition and the property $\|A^2\|_2\leq \|A\|^2_{2}$ of the matrix norm, we can write
\begin{eqnarray*}
    ||\mc{A}^2||_{M} &=& \displaystyle{\max_{1\leq i\leq p}(\|\tilde{\mc{A}^2}(:,:,i)\|_2)}=\displaystyle{\max_{1\leq i\leq p}(\|\tilde{\mc{A}}(:,:,i)^2\|_2)}\leq \displaystyle{\max_{1\leq i\leq p}(\|\tilde{\mc{A}}(:,:,i)\|_2)}\displaystyle{\max_{1\leq i\leq p}(\|\tilde{\mc{A}}(:,:,i)\|_2)}\\
    &=& ||\mc{A}||_{M}||\mc{A}||_{M}= ||\mc{A}||_{M}^2
\end{eqnarray*}
\noindent (ii)~Let $\|\mc{A}\|_M<1$. Then by part (i), we obtain $\|\displaystyle\lim_{s\rightarrow\infty}\mc{A}^s\|_M\leq\displaystyle\lim_{s\rightarrow\infty}\|\mc{A}\|_M^s=0.$
Therefore, $\displaystyle\lim_{s\rightarrow\infty}\mc{A}^s=\mc{O}.$
To verify the second statement in (ii), assume $\rho(\mc{A})<1.$
Based on the singular value decomposition \cite{jin2023}, the tensor $\mc{A}^s$ is rewritten as $\mc{A}^s=\mc{U}\m \mc{\D}\m \mc{V}^*,$ such that $\mc{U},~\mc{V}\in \mathbb{C}^{n\times n \times n_3}$ are unitary and $\mc{D}\in\mathbb{C}^{n\times n \times n_3}$ is diagonal with diagonal positions of $\mc{D}\in\mathbb{C}^{n\times n \times n_3}$ filled by the eigenvalues of $\mc{A}^s.$
Hence $\displaystyle\lim_{s\rightarrow\infty}\mc{A}^s=\mc{O}$ if and only  $|\tilde{\mc{\D}}(:,:,i)| < 1$ for each $i=\overline{1,p}$.
This completes the part (ii).

\noindent (iii)~ It suffices to verify the necessary part, because only if part follows from definition \ref{tenseries}.
Let $\displaystyle\lim_{s\rightarrow\infty}\mc{A}^s=\mc{O}.$
Following the part (ii) of this theorem, $\rho(\mc{A})<1$, it is clear that all eigenvalues of $(\mc{I}-\mc{A})$ are nonzero.
This implies nonsingularity of $(\mc{I}-\mc{A})$.
It is now available
\begin{equation}\label{eq1212}
   (\mc{I}+\mc{A}+\mc{A}^2+\cdots+\mc{A}^s)\m(\mc{I}-\mc{A})=\mc{I}-\mc{A}^{s+1}.
\end{equation}
Post-multiplication of Eq. \eqref{eq1212} by $(\mc{I}-\mc{A})^{-1}$ yields
$\mc{I}+\mc{A}+\mc{A}^2+\cdots+\mc{A}^s=(\mc{I}-\mc{A}^{s+1})\m(\mc{I}-\mc{A})^{-1}.$
The limit $s\rightarrow\infty$ leads to $\displaystyle\sum_{s=0}^\infty\mc{A}^s=(\mc{I}-\mc{A})^{-1}.$
\end{proof}

\begin{algorithm}[H]
\caption{Higher order Gauss-Seidel method based on $M$-product} \label{alg:mgs}
\begin{algorithmic}[1]
\Procedure{Gauss-Seidel}{$\mc{A},\mc{B},\epsilon, \mbox{MAX}$}
\State {\bf Input} $\mc{A} \in \mbc^{m \times m\times p}, \mc{B} \in \mbc^{m \times 1\times p}$ and  $M\in\mbc^{p\times p}$.
\State Compute $\tilde{\mc{A}}=\mc{A}\times_3 M$
\For{$i=1$ to $p$}
\State Compute $\tilde{\mc{L}}(:,:,i)=lowerdiag(\tilde{\mc{A}}(:,:,i)),~\tilde{\mc{U}}(:.:,i)=\tilde{\mc{A}}(:,:,i)-\tilde{\mc{L}}(:,:,i)-diag(\tilde{\mc{A}}(:,:,i))$
\State Compute $\tilde{\mc{T}}(:,:,i)=-(\tilde{\mc{L}})^{-1}(:,:,i)\tilde{\mc{U}}(:.:,i)$ and
$\mc{C}(:,1,i)=(\tilde{\mc{L}})^{-1}(:,:,i)\mc{B}(:,1,i)$
\State Initial guess $\tilde{\mc{X}}^{0}(:,1,i)$
\For{$s=1$ to MAX}
\State $\tilde{\mc{X}}^{s}(:,1,i)=\tilde{\mc{T}}(:,:,i)\tilde{\mc{X}}^{s-1}(:,1,i)+\mc{C}(:,1,i)$
\If
{$\|\tilde{\mc{X}}^{s}(:,1,i)-\tilde{\mc{X}}^{0}(:,1,i)\|\leq \epsilon$}
\State \textbf{break}
\EndIf
\State $\tilde{\mc{X}}^{0}(:,1,i)\leftarrow \tilde{\mc{X}}^{s}(:,1,i)$
\EndFor
\EndFor
 \State Compute $\mc{X}^{s}=\tilde{\mc{X}}^{s}\times_3M^{-1}$
\State \Return $\mc{X}^{s}$
\EndProcedure
 \end{algorithmic}
\end{algorithm}

\begin{theorem}\label{th610}
Consider $\mc{A}\in\mathbb{C}^{m \times m \times p}$ and $M\in\mathbb{C}^{p \times p}$.
The tensor splitting based iterative method
\begin{equation}\label{ITS1}
\mc{X}^{s}=\mc{T}\m\mc{X}^{s-1}+\mc{C} ~~\text{for}~~ s=1,2,\ldots
\end{equation}
converges to $\mc{A}^{-1}\m\mc{B}$  for arbitrary initialization $\mc{X}^{0}$ if and only if $\rho(\mc{T}) < 1$.
\begin{proof}
Without loss of general principle, choose the initial guess $\mc{X}^{0}=\mc{O}$.
Then by \eqref{ITS}, we obtain $\mc{X}^{1}=\mc{C}$.
This leads $\mc{X}^{2}=\mc{T}\m\mc{X}^{(1)}+\mc{C}= \mc{T}\m\mc{C}+\mc{C}=(\mc{T}+\mc{I})\m\mc{C}$.
Continuing the same procedure $(s+1)$ times, we obtain
$$\mc{X}^{s+1}=(\mc{I}+\mc{T}+\mc{T}^{2}+\cdots +\mc{T}^{s})\m\mc{C}.$$
Following Theorem \ref{th69} (iii) with  $s\rightarrow \infty$, we get
\begin{equation*}
\lim_{s\rightarrow \infty}
\mc{X}^{s+1} = (\mc{I}-\mc{T})^{-1}\m\mc{C} \Longleftrightarrow \rho(\mc{T})\leq 1,
\end{equation*}
or equivalently, $\displaystyle\lim_{s\rightarrow\infty}\mc{X}^{s+1}=\mc{A}^{-1}\m\mc{B}\Longleftrightarrow\rho(\mc{T})< 1$.
\end{proof}
\end{theorem}

The following results are derived as corollaries in view of the above theorems.

\begin{corollary}
Assume $\mc{A}\in\mathbb{C}^{m \times m \times p}$ and $M\in\mathbb{C}^{p \times p}$.
If the inequality $||\mc{T}||< 1$ holds, then the iterative sequence \eqref{ITS} converges to $\mc{A}^{-1}\m\mc{B}$ for arbitrary initial approximation $\mc{X}^{0}$.
\end{corollary}

\begin{corollary}
If $\mc{A}\in\mathbb{C}^{m \times m \times p}$ is diagonally dominant and $M\in\mathbb{C}^{p \times p}$, then it follows:
    \begin{itemize}
   \item[\rm (a)] Jacobi iterations converge for arbitrary $\mc{X}^{(0)}$.
   \item[\rm (b)] Gauss-Seidel iterations converge for arbitrary $\mc{X}^{(0)}$.
    \end{itemize}
\end{corollary}

In Tables \ref{Jtable} and \ref{GStable}, we compare the mean CPU time (MT$^M$) when computed under $M$-product again the matrix structured computation mean CPU time (MT). The number of iterations (IT$^M$) it takes in $M$-product and matrix structured computation (denoted by IT) are also presented.

\begin{table}[H]
\begin{center}
\caption{ Comparison analysis of CPU-time, residual errors for Jacobi method for different order tensors and matrices with taking $\epsilon=10^{-10}$ }
\vspace{0.2cm}
        \renewcommand{\arraystretch}{1.1}
        \begin{tabular}{cccccccc}
    \hline
        Size of $\mc{A}$& IT$^M$  & MT$^{M}$  &$\|\mc{A}\m\mc{X}-\mc{B}\|_M$ & Size of $A$ & IT & MT & $\|AX-b\|_F$\\
           \hline
  $100\times 100\times 400$& 88&0.26 &$4.96e^{-09}$ &$2000\times 2000$ &96 & 43.04&$9.15e^{-08}$\\
    \hline
  $200\times 200\times 400$& 88 &0.81 &$9.97e^{-09}$ &$4000\times 4000$ &101 & 71.56&$7.78e^{-08}$\\
  \hline
  $300\times 300\times 400$& 89& 1.01&$1.49e^{-08}$ &$6000\times 6000$ &93& 8275&$2.40e^{-07}$\\
    \hline
  $400\times 400\times 400$& 89 & 1.80&$2.00e^{-08}$ &$8000\times 8000$ & 99 & 19466&$3.41e^{-07}$\\
  \hline
    \end{tabular}\label{Jtable}
\end{center}
\end{table}

\begin{table}[H]
\begin{center}
\caption{ Comparison analysis of CPU-time, residual errors for Gauss-Seidel method for different order tensors and matrices with $\epsilon=10^{-10}$ }
\vspace{0.2cm}
        \renewcommand{\arraystretch}{1.1}
        \begin{tabular}{cccccccc}
        \hline
        Size of $\mc{A}$& IT$^M$  & MT$^{M}$  &$\|\mc{A}\m\mc{X}-\mc{B}\|_M$ & Size of $A$ & IT & MT & $\|AX-b\|_F$\\
    \hline
  $100\times 100\times 400$& 15&0.21 &$1.80e^{-09}$ &$2000\times 2000$ &16 & 8.33&$1.73e^{-08}$\\
    \hline
  $200\times 200\times 400$& 15&0.60&$3.43e^{-09}$ &$4000\times 4000$ &17 & 389.57&$7.33e^{-08}$\\
  \hline
  $300\times 300\times 400$& 15 & 0.76&$5.33e^{-09}$ &$6000\times 6000$ & 17 & 1227.65&$1.92e^{-07}$\\
    \hline
  $400\times 400\times 400$& 15& 1.12&$7.00e^{-09}$ &$8000\times 8000$ &17 &3506.91 &$3.89e^{-07}$\\
  \hline
    \end{tabular}\label{GStable}
\end{center}
\end{table}

\subsection{Tikhonov's regularization}

Next, we define the well-posed multilinear system in the sense of Hardamard.

\begin{definition}
 Let $\mc{A} \in \mbc^{m\times n\times p}$ and  $M\in\mbc^{p\times p}$. Then the multilinear system $\mc{A}\m\mc{X}=\mc{B}$ is called well-posed if there exist unique stable solution.
 Otherwise, this system is ill-posed.
 \end{definition}

Further, we define $\mc{A}$ is ill-conditioned if at least one of the frontal slices of $\tilde{\mc{A}}$ (or $\mc{A}\times_3M$) is ill-conditioned matrix.
Thus for  ill-posed multilinear system $\mc{A}\m\mc{X}=\mc{B}$ or when $\mc{A}$ is ether ill-conditioned or singular, we will apply regularization in the following manner while solving
 \begin{equation}\label{eqmult}
  \mc{A}\m\mc{X}=\mc{B}.
\end{equation}
We first normalize \eqref{eqmult} as $\mc{A}^*\m\mc{A}\m\mc{X}= \mc{A}^*\m\mc{B}$ and then modify the multilinear system based on the Tikhonov's regularization \cite{tikhonov} as follows.\begin{equation}\label{eqregul}
\mc{T}_{\lambda}\m\mc{X}:=(\mc{A}^*\m\mc{A}+\lambda \mc{I})\m\mc{X} = \mc{A}^*\m\mc{B},
\end{equation}
such that  $\lambda (>0)$ is the selected regularization parameter.
There are several regularization methods available for ill-posed linear systems, and those methods can be studied in the future in the framework of multilinear systems.
The solution to \eqref{eqregul} can be computed by Algorithm \ref{alg:mtik}.

\begin{algorithm}[H]
\caption{Tikhonov's regularized solution based on $M$-product} \label{alg:mtik}
\begin{algorithmic}[1]
\State {\bf Input} $\mc{A} \in \mbc^{m \times n\times p}, \mc{B} \in \mbc^{n \times 1\times p}$ and  $M\in\mbc^{p\times p}$.
\State Compute $\mc{T}_{\lambda}=\mc{A}^*\m\mc{A}+\lambda \mc{I},~\tilde{\mc{T}_{\lambda}}=\mc{T}_{\lambda}\times_3 M$, $\mc{C}=\mc{A}^*\m\mc{B}$
\For{$i=1$ to $p$}
\State Compute $\tilde{\mc{X}_{\lambda}}(:,1,i)=\tilde{\mc{T}_{\lambda}}^{-1}(:,:,i)\m\mc{C}(:,1,i)$
\EndFor
 \State Compute $\mc{X}_{\lambda}=\tilde{\mc{X}_{\lambda}}\times_3M^{-1}$
\State \Return $\mc{X}_{\lambda}$
 \end{algorithmic}
\end{algorithm}
In the case of matrices, Barata and Hussein \cite{barata2012moore} proved the following result:
\begin{proposition}\label{tik-mat}
   Let $A\in\mathbb{C}^{m\times n}$. Then $\displaystyle\lim_{\lambda\to 0}(A^*A+\lambda I)A^*=A^{\dagger}$.
\end{proposition}
Using Proposition \ref{tik-mat} in the frontal slice of $\tilde{\mc{A}}$, we can show that the solution $(\mc{A}^*\m\mc{A}+\lambda \mc{I})^{-1}\m\mc{A}^*\m\mc{B}$ of equation \ref{eqregul} converges to the least square solution $\mc{A}^{\dagger}\m\mc{B}$ of equation \eqref{eqmult}.

\subsubsection{Image deblurring}

In this subsection, we investigate the reconstruction of color images from blurred images in the tensor structured domain with the help of $M$-product. The authors of \cite{calvetti1,hansen} discussed the grayscale image deblurring problem in the following ill-posed linear system.
\begin{equation*}
A\vec{x}=\vec{b},~A\in\mathbb{R}^{n^2\times n^2},~\vec{x}\in\mathbb{R}^{n^2\times 1},~\mbox{ and }\vec{b}\in\mathbb{R}^{n^2\times 1},
\end{equation*}
where $A$ is the blurring matrix, $\vec{x}$ is the vectorization of the true image, generated by stacking the columns of the true image $X \in \mathbb{R}^{n \times n}$ one on the other.
Similarly,  $\vec{b}$ is obtained from a blurred image (often corrupted with noise) $B\in \mathbb{R}^{n \times n}$.
Such a blurring process was modeled in \cite{Kernfeldlinear} using the equation
\begin{equation*}
    \mc{A}\m \mc{X}=\mc{B},
\end{equation*}
where $\mc{A}\in \mathbb{R}^{n\times n \times n} $ is the known blurring tensor, which is constructed using a banded Toeplitz matrix.
The same authors in \cite{Kernfeldlinear} considered $\mc{B}\in  \mathbb{R}^{n\times 1\times n}$  and  $\mc{X}\in  \mathbb{R}^{n\times 1\times n}$ using $\mat$ and $\mat^{-1}$.
Now we consider the color image deblurring problem in the form of the tensor equation
\[
    \mc{A}\m \mc{X}=\mc{B},  ~\text{~where~}~\mc{A}\in \mathbb{R}^{n\times n \times 3},~ {\mc{X}\in \mathbb{R}^{n\times n \times 3}, ~\mc{B}\in \mathbb{R}^{n\times n \times 3}}.
\]
Here, the true image is tensor $\mc{X}$, which we must find, and the blurring tensor we need to construct is $\mc{A}$.
Frontal slices of the blurring tensor $\mc{A}^{n \times n \times 3}$ are given by
$$ \mc{A}(i,j,k) =\left\{\begin{array}{ll}
\frac{1}{\sigma \sqrt{2 \pi}} e^{-\frac{(i-j)^{2}}{2 \sigma^{2}}}, & |i-j| \leq b\\
0, & \text { otherwise, }
\end{array}\right.$$
for $k=1,2,3$.
Here, $b$ is the bandwidth, and $\sigma$ controls the amount of smoothing, i.e., the problem is ill-posed when $\sigma$ is larger.

For the numerical experiment, two color images of sizes  $256\times256\times 3 $ and $512\times512\times 3 $ were  considered.
Using $\sigma =4$ and $b=30$, we generate the blurred image $\mc{B} = \mc{A}\m\mc{X} + \mc{N}$, where $\mc{N}$ is a noise tensor, that is constructed by taking each slice as Gaussian noise with the mean $0$ and variance $10^{-3}$.
The true images  are provided in Figure \ref{im1} (a) and Figure \ref{im2} (a), while blurred and noisy images are shown in Figure \ref{im1} (b) and Figure \ref{im2} (b), respectively.
Finally, the regularized approach as described in Algorithm \ref{alg:mtik} can be used to reconstruct of the true image.
The resulting restored images are  shown in Figure \ref{im1} (c) and Figure \ref{im2} (c), respectively.

\begin{figure}[H]
\begin{center}
\subfigure[]{\includegraphics[height=4.6cm]{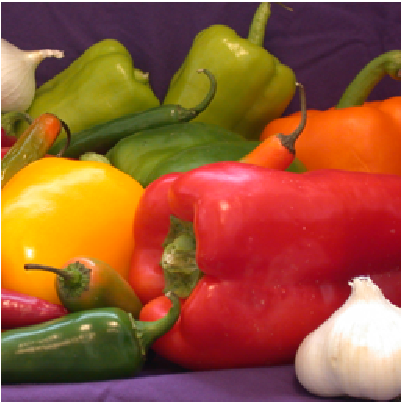}}~~~~
\subfigure[]{\includegraphics[height=4.6cm]{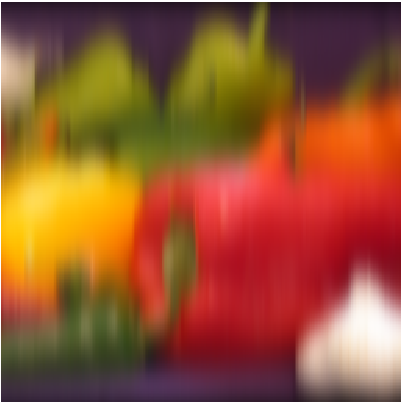}}~~~~
\subfigure[]{\includegraphics[height=4.6cm]{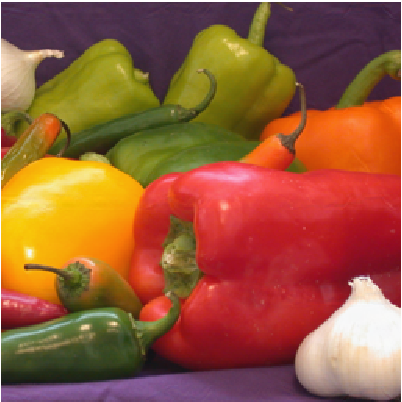}}
\hspace*{-1.05cm}\caption{(a) True image of size $256\times 256$.~~~(b) Blurred noisy image.~~~~~ (c) Reconstructed image.
}\label{im1}
\end{center}
\end{figure}

\begin{figure}[H]
\begin{center}
\subfigure[]{\includegraphics[height=4.6cm]{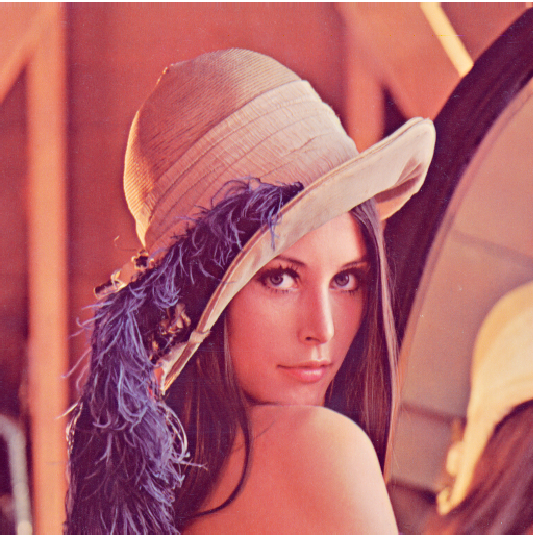}}~~~~
\subfigure[]{\includegraphics[height=4.6cm]{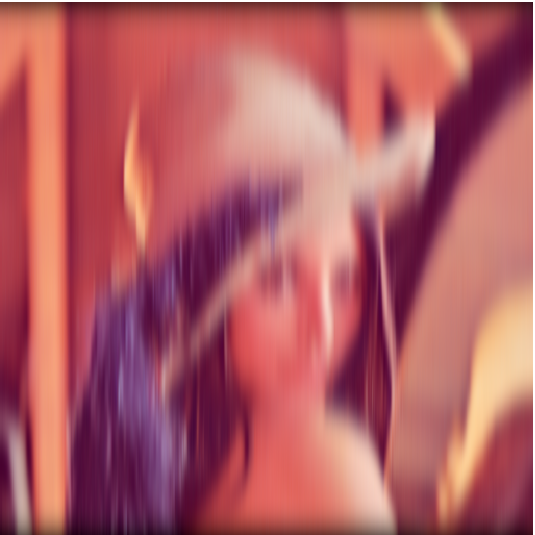}}~~~~
\subfigure[]{\includegraphics[height=4.6cm]{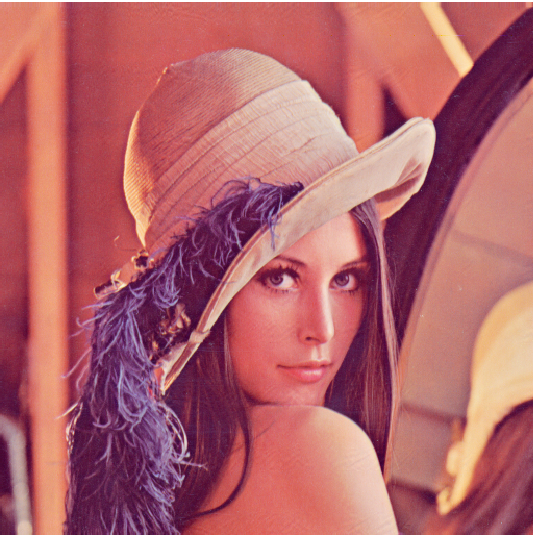}}
\caption{(a) True image of size $512\times 512$~~~~~~(b) Blurred noisy image.~~~~ (c) Reconstructed image.
}\label{im2}
\end{center}
\end{figure}

\section{Conclusion}
The concept of the Drazin, core-EP, DMP, MPD, and CMP  inverses for the third order tensor under the $M$-product have been studied along with a few characterizations of these inverses.
Solutions to multilinear systems based on theses inverses are also discussed.
High order Jacobi and Gauss-Seidel methods have been developed to solve such multilinear systems.
In addition, a few numerical examples and an application to image deblurring problem are also discussed.
The following problems can be considered in future research:
\begin{itemize}
    \item To develop iterative-based algorithms for computing the MP, Drazin, core-EP inverse under $M$-product.
    \item Investigate the perturbation bounds for the proposed generalized inverses under the $M$-product.
\item To study of other composite generalized inverses such as MPCEP, CEPMP  and bilateral inverses based on $M$-product.

\end{itemize}

\section*{Funding}
\begin{itemize}
    \item Ratikanta Behera acknowledges the support of the Science and Engineering Research Board (SERB), Department of Science and Technology, India, under Grant No. EEQ/2022/001065.
\item Jajati Keshari Sahoo is supported by the Science and Engineering Research Board (SERB), Department of Science and Technology, India, under Grant No. SUR/2022/004357.
\item Predrag Stanimirovi\' c is supported by the Science Fund of the Republic of Serbia (No. 7750185, Quantitative Automata Models: Fundamental Problems and Applications - QUAM) and
by the Ministry of Education, Science and Technological Development, Republic of Serbia, grant no. 451-03-65/2024-03/200124.
\end{itemize}

\section*{Conflict of Interest}
The authors declare no potential conflict of interest.

\section*{Data Availability}
Data sharing not applicable to this article as no datasets were generated or analyzed during the current study.

\section*{ORCID}
Jajati Keshari Sahoo~\orcidC \href{https://orcid.org/0000-0001-6104-5171}{ \hspace{2mm}\textcolor{lightblue}{https://orcid.org/0000-0001-6104-5171}} \\
Ratikanta Behera~\orcidA \href{https://orcid.org/0000-0002-6237-5700}{ \hspace{2mm}\textcolor{lightblue}{ https://orcid.org/0000-0002-6237-5700}}\\
Predrag S. Stanimirovi\'c~\orcidD \href{https://orcid.org/0000-0003-0655-3741 }{\hspace{2mm}\textcolor{lightblue}{https://orcid.org/0000-0003-0655-3741 }} \\

\bibliographystyle{abbrv}
\bibliography{Reference}
\end{document}